\documentclass[a4paper,12pt, reqno]{amsart}

\usepackage{ amssymb, amsmath, enumerate, amsfonts, amsthm, mathrsfs, url, bm, mathtools}

\usepackage{times}

\usepackage{xcolor}  	
\usepackage[backref]{hyperref}
\hypersetup{
	colorlinks,
    linkcolor={blue!60!black},
    citecolor={blue!60!black},
    urlcolor={red!60!black}
}

\usepackage{color}

\numberwithin{equation}{section}

 \newcommand{\comment}[1]{}  

\usepackage[margin=1in]{geometry}

\newcommand{\N}{\mathbb{N}}
\newcommand{\Z}{\mathbb{Z}}

\newcommand{\R}{\mathbb{R}}
\newcommand{\C}{\mathbb{C}}
\newcommand{\T}{\mathbb{T}}
\newcommand{\F}{\mathbb{F}}

\newcommand{\set}[1]{\mathopen{}\left\{#1\mathclose{}\right\}}

\newcommand{\abs}[1]{\mathopen{}\left| #1\mathclose{}\right|}
\newcommand{\bigabs}[1]{\bigl| #1 \bigr|}
\newcommand{\Bigabs}[1]{\Bigl| #1 \Bigr|}
\newcommand{\biggabs}[1]{\biggl| #1 \biggr|}

\newcommand{\sqbrac}[1]{\mathopen{}\left[ #1 \mathclose{}\right]}

\newcommand{\Bigsqbrac}[1]{\Bigl[ #1 \Bigr]}
\newcommand{\biggsqbrac}[1]{\biggl[ #1 \biggr]}

\newcommand{\ceil}[1]{\mathopen{}\left\lceil #1 \mathclose{}\right\rceil}

\newcommand{\floor}[1]{\mathopen{}\left\lfloor #1 \right\rfloor}
\newcommand{\brac}[1]{\mathopen{}\left( #1 \mathclose{}\right)}
\newcommand{\bigbrac}[1]{\bigl( #1 \bigr)}
\newcommand{\Bigbrac}[1]{\Bigl( #1 \Bigr)}
\newcommand{\biggbrac}[1]{\biggl( #1 \biggr)}

\newcommand{\norm}[1]{\mathopen{}\left\| #1\mathclose{}\right\|}
\newcommand{\bignorm}[1]{\big\| #1 \big\|}

\newcommand{\ang}[1]{\mathopen{}\left\langle#1\mathclose{}\right\rangle}
\newcommand{\bigang}[1]{\bigl\langle #1 \bigr\rangle}

\newcommand{\recip}[1]{\frac{1}{#1}}
\newcommand{\trecip}[1]{\tfrac{1}{#1}}

\newcommand{\vx}{\underline{x}}

\newcommand{\vh}{\underline{h}}

\newcommand{\vd}{\underline{d}}

\newcommand{\vv}{\underline{v}}

\newcommand{\E}{\mathbb{E}}

\renewcommand{\Re}{\operatorname{Re}}

\newcommand{\intd}{\mathrm{d}}
\newcommand{\supp}{\mathrm{supp}}

\newcommand{\eps}{\varepsilon}
\newcommand{\hash}{\#}

\makeatletter
\let\@@pmod\pmod
\DeclareRobustCommand{\pmod}{\@ifstar\@pmods\@@pmod}
\def\@pmods#1{\mkern4mu({\operator@font mod}\mkern 6mu#1)}
\makeatother

\newtheorem{theorem}{Theorem}[section]

\newtheorem{corollary}[theorem]{Corollary}

\newtheorem{lemma}[theorem]{Lemma}

\theoremstyle{definition}

\newtheorem{remark}[theorem]{Remark}

\numberwithin{theorem}{section}

\renewcommand{\leq}{\leqslant}
\renewcommand{\geq}{\geqslant}

\begin{document}

\title
{Bounds in a popular multidimensional nonlinear Roth theorem}

\author{Sarah Peluse}
\address{Department of Mathematics\\
University of Michigan
}
\email{speluse@umich.edu}

\author{Sean Prendiville}
\address{School of Mathematical Sciences\\
Lancaster University
}
\email{s.prendiville@lancaster.ac.uk}

\author{Xuancheng Shao}
\address{Department of Mathematics\\
University of Kentucky
}
\email{xuancheng.shao@uky.edu}

\date{\today}

\begin{abstract}
A nonlinear version of Roth's theorem states that dense sets of integers contain configurations of the form $x$, $x+d$, $x+d^2$. We obtain a multidimensional version of this result, which can be regarded as a first step towards effectivising those cases of the multidimensional polynomial Szemer\'edi theorem involving polynomials with distinct degrees. In addition, we prove an effective ``popular'' version of this result, showing that every dense set has some non-zero $d$ such that the number of  configurations with difference parameter $d$ is almost optimal. Perhaps surprisingly, the quantitative dependence in this result is exponential, compared to the tower-type bounds encountered in the popular linear Roth theorem. 
\end{abstract}

\maketitle

\setcounter{tocdepth}{1}
\tableofcontents

\section{Introduction}\label{sec:intro}
A long-standing programme in additive combinatorics concerns  obtaining effective
versions of the multidimensional Szemer\'edi theorem \cite{FurstenbergKatznelsonErgodic}
and the polynomial Szemer\'edi theorem \cite{BergelsonLeibmanPolynomial}, as well as their
common generalisation (also found  in \cite{BergelsonLeibmanPolynomial}); see for instance
\cite{GowersNewFour,GowersEnduring}. In this paper, we take a first step towards effectivising those cases of the multidimensional polynomial Szemer\'edi theorem involving polynomials of distinct degrees.
\begin{theorem}[Density bound]\label{thm:density-bound}
If $A \subset \set{1,2,\dots,N}\times \set{1,2,\dots,N}$ does not contain a triple of the form 
\begin{equation}\label{eq:config}
(x,y),\ (x+d, y),\ (x,y+d^2)\quad \text{with}\quad d \neq 0,
\end{equation}
then
$$
|A| = O\brac{ {N^2}/{(\log N)^{c}}}.
$$
Here $c>0$ is an absolute constant\footnote{We expect that $c = 2^{-300}$ is permissible.}.
\end{theorem}
Simultaneous to our work, Kravitz, Kuca and Leng \cite{kravitz2024corners} have obtained a result of this type for configurations of the form $(x,y)$, $(x+P(d),y)$, $(x,y+P(d))$, where $P\in\Z[d]$ has an integer root of multiplicity one. Together, these constitute the first effective cases of the multidimensional polynomial Szemer\'edi theorem
involving both a genuinely multidimensional configuration and a genuinely nonlinear
polynomial.

We derive our density bound on taking $\eps = \delta^3/2$ in the following stronger result, which proves the existence of a ``popular''  difference (with effective bounds).
\begin{theorem}[Popular difference in two dimensions]\label{thm:2D-pop-diff}
Let $A \subset \set{1,2,\dots,N}\times \set{1,2,\dots,N}$ with $|A| \geq \delta N^2$ and let $0<\eps \leq 1/2$.
Either $N \leq \exp\brac{\eps^{-O(1)}}$ or there exists $d \neq 0$ such that
$$
\hash\set{(x,y) \in A : (x+d, y), (x,y+d^2) \in A} \geq \brac{\delta^3-\eps} N^2.
$$
\end{theorem}
Theorem \ref{thm:2D-pop-diff} is a consequence of our more general Theorem \ref{thm:pop-com-diff}, which  we also use (in \S\ref{sec:pop-com-diff}) to derive the existence of a popular difference in the one-dimensional context:
\begin{theorem}[Popular difference in one dimension]\label{thm:1D-pop-diff}
Let $A \subset \set{1,2,\dots, N}$ with $|A| \geq \delta N$ and let $0<\eps \leq 1/2$.
Either $N \leq \exp\brac{\eps^{-O(1)}}$ or there exists $d \neq 0$ such that
$$
\hash\set{x \in A : x+d, x+d^2 \in A} \geq \brac{\delta^3-\eps} N.
$$
\end{theorem}
We remark that the above exponential bound $N \leq \exp\brac{\eps^{-O(1)}}$ is in sharp contrast to the tower-type bound necessary for linear three-term progressions, see \cite{FPZTower}. 

As a corollary to Theorem \ref{thm:1D-pop-diff}, we obtain a new proof that subsets of
$\set{1,2,\dots,N}$ lacking the nonlinear Roth configuration have at most polylogarithmic
density, which is the main result of \cite{PelusePrendivillePolylogarithmic}. As
in~\cite{PelusePrendivillePolylogarithmic}, our argument crucially relies on the advances
made in~\cite{PelusePrendivilleQuantitative}.

The main input to all of the above results is the following inverse theorem, which characterises those functions that count many configurations of the form \eqref{eq:config}. Importantly, the quantitative dependence in this inverse theorem is polynomial. (In our inverse theorem, it turns out to be more natural  to work with the asymmetric grid\footnote{$[N]:=\set{1,2,\dots,N}$.} $[N]\times [N^2]$, rather than $[N]\times [N]$.) 
\begin{theorem}[Inverse theorem]\label{thm:inverse}
Let $f_0, f_1, f_2 : \Z^2 \to \C$ be 1-bounded functions with support in $[N]\times [N^2]$. Suppose that
\begin{equation}\label{eq:large-count}
\biggabs{\sum_{x,y,d} f_0(x,y)f_1(x+d,y) f_2(x,y+d^2)}\\ \geq \delta N^4.
\end{equation}
Then\footnote{$f\ll g$ is a shorthand for $f = O(g)$.} either $N \ll \delta^{-O(1)}$ or there exist\footnote{$\T:=\R/\Z$, $\norm{\alpha}_{\T}:=\min_{n\in\Z} |\alpha - n|$, $e(\alpha) := e^{2\pi i \alpha}$.} $\alpha, \beta \in \T$ and $q \ll \delta^{-O(1)}$ with $\norm{q\alpha}_{\T} \ll\delta^{-O(1)}/N$ and $\norm{q\beta}_{\T} \ll\delta^{-O(1)}/N^2$ such that
\begin{equation}\label{eq:inverse-eq}
\sum_{y} \biggabs{\sum_{x} f_1(x, y) e\brac{\alpha x}} \gg \delta^{O(1)} N^3 \quad \text{and} \quad\sum_{x} \biggabs{\sum_{y} f_2(x, y) e\brac{\beta y}} \gg \delta^{O(1)} N^3  .
\end{equation}
\end{theorem}
This inverse theorem is proved in \S\ref{sec:inv-thm}, where it is derived from a combination of Theorem \ref{thm:inv-thm-1} and Theorem \ref{thm:inv-thm-2}.

The analogous one-dimensional inverse theorem \cite{PelusePrendivilleQuantitative,
  PelusePrendivillePolylogarithmic} has been used by Krause, Mirek and Tao
\cite{KrauseMirekTaoPointwise} in their breakthrough work on pointwise convergence of
bilinear polynomial ergodic averages. We expect our multidimensional inverse theorem to be useful in the study of ergodic averages of the form 
$$
\recip{N}\sum_{d=1}^N f(T_1^dx)g(T_2^{d^2} x),
$$
where $T_1$ and $T_2$ commute.
\begin{remark}[The obstructions in our inverse theorem are necessary]
By removing the absolute value in \eqref{eq:inverse-eq}, at the cost of introducing signs, one can re-phrase our inverse theorem as saying that if the counting operator for our configuration is large, as in \eqref{eq:large-count}, then the function weighting the $(x+d,y)$ term correlates with a function of the form
$$
(x,y) \mapsto e(\alpha x)\phi(y),
$$
where $\alpha$ is ``major arc at scale $N$'' and $\phi : \Z \to \C$ is 1-bounded. Similarly, the function weighting the $(x, y+d^2)$ term correlates with a function of the form
$$
(x,y) \mapsto \psi(x)e(\beta y),
$$
where $\beta$ is major arc at scale $N^2$ and $\psi : \Z \to \C$ is 1-bounded.

Notice that the above formulation does not posit any structure for the 1-bounded functions $\phi$ and $\psi$. There is no hope in being able to extract further structure regarding these functions: their arbitrariness is necessary. This is exhibited in the following example.

Let $\phi : [N^2]\to\set{\pm 1}$ be a random function that takes the value 1 with probability 1/2, independently as the argument ranges over $[N^2]$. For $(x,y) \in [N]\times [N^2]$ set
$$
f_0(x,y) := \phi(x)\phi(y)(-1)^{x+y},\quad f_1(x,y) := (-1)^x\phi(y)\quad\text{and}\quad f_2(x,y):=\phi(x)(-1)^y.
$$
Then
\begin{equation*}
\sum_{x,y,d} f_0(x,y)f_1(x+d,y) f_2(x,y+d^2)= \sum_{x,y,d} 1_{[N]}(x)1_{[N^2]}(y)1_{[N]}(x+d) 1_{[N^2]}(y+d^2) \gg N^4.
\end{equation*}
Yet for any ``reasonable'' notion of structured function $\psi$, the randomness of $\phi$ means that the fibre maps $y \mapsto f_1(x,y)$ and $x\mapsto f_2(x,y)$ almost surely do not correlate with $\psi$.
\end{remark}

\subsection{Previous work}
\subsubsection{Density bounds for higher-dimensional configurations}

Let us first put our density bound (Theorem \ref{thm:density-bound}) in context.

Given any two-point polynomial configuration (without congruence obstructions), it is
likely that one can show that subsets of $[N]^n$ lacking this configuration have at most
polylogarithmic density. One way to prove this is to generalise S\'ark\"ozy's argument
\cite{SarkozyDifferenceI} for the configuration $x,x+d^2$; see Lyall and Magyar's
\cite{LyallMagyarPolynomial} treatment of $\vx, \vx + (d,d^2,\dots, d^n)$ for an instance
of this approach. It is also plausible that even stronger bounds could be obtained by
adapting the argument of Bloom and Maynard~\cite{BloomMaynard}, who have proved the current best
known bounds in S\'ark\"ozy's theorem.

Longer linear configurations in $\Z^n$ can be encoded as \textit{matrix progressions}
$$
\vx,\ \vx + T_1\vd,\ \dots\ \vx + T_s \vd,
$$
where $T_i$ are $n\times n$ matrices with integer entries. As was demonstrated in
\cite{PrendivilleMatrix}, provided that the $T_i$ and their differences $T_i - T_j$
($i\neq j$) are all non-singular, one can extract density bounds for sets lacking these
configurations by adapting the higher-order Fourier analysis techniques of Gowers \cite{GowersNew}. One may be able to synthesise this approach with that of \cite{PrendivilleQuantitative} in order to bound the density of sets lacking nonlinear matrix progressions of the form
$$
\vx,\ \vx + T_1(d_1^k,\dots, d_n^k),\ \dots\ \vx + T_s (d_1^k,\dots, d_n^k),
$$ 
again provided that the $T_i$ and their differences $T_i - T_j$ ($i\neq j$) are all non-singular.

It is much harder to obtain quantitative bounds for sets lacking \textit{singular} matrix progressions. Preceding this article and that of Kravitz-Kuca-Leng \cite{kravitz2024corners}, there are  only two results in this direction (for progressions of length greater than two). The first is due to Shkredov \cite{ShkredovGeneralisation} and obtains double-logarithmic bounds for sets lacking \textit{any} three-point matrix progression, regardless of the singularity of the matrices and their differences. This very general result is deduced in \cite[Theorem B.2]{PrendivilleMatrix} from Shkredov's much more concrete result which bounds the density of sets lacking two-dimensional corners:
\begin{equation}\label{eq:corners}
(x,y),\ (x+d, y),\ (x,y+d) \quad \text{with}\quad d\neq 0.
\end{equation} 
Very recently, the first author \cite{peluse2022subsets} obtained the first reasonable bounds for sets lacking a specific four-point singular matrix progression, namely L-shapes:
\begin{equation}\label{eq:L-shapes}
(x,y),\ (x+d, y),\ (x,y+d),\ (x,y + 2d) \quad \text{with}\quad d\neq 0.
\end{equation}
Although this argument was carried out over the finite field model $\F_p^n\times \F_p^n$, those conversant with the appropriate machinery will see that adapting the argument to the integer setting is now within reach, albeit technically demanding.

We believe the arguments of this paper open up the possibility of obtaining bounds for sets lacking polynomial progressions of the form
\begin{equation}\label{eq:long-config}
\vx,\ \vx + P_1(d)\vv_1,\ \dots \ ,\ \vx + P_s(d) \vv_s \quad \text{with} \quad d \neq 0,
\end{equation}
where $\vv_1,\dots, \vv_s$ are fixed integer vectors and the polynomials $P_i$ have
integer coefficients, zero constant term and (importantly) \textit{distinct} degrees. Our
configuration \eqref{eq:config} corresponds to taking $P_1 = d$, $P_2 = d^2$, $\vv_1 =
(1,0)$ and $\vv_2 = (0,1)$. For the analogous problem over the finite field $\F_p$, our
paper corresponds to work of Han-Lacey-Yang \cite{han2021polynomial}, and we posit that it
is possible to obtain the analogue of Kuca's more general result \cite{kuca2024multidimensional} in the integers. In this paper we make heavy use of the one-dimensional machinery developed in \cite{PelusePrendivilleQuantitative} for $x,x+d,x+d^2$, and for \eqref{eq:long-config} we expect that the appropriate one-dimensional machinery is available in \cite{PeluseBounds}. 
\subsubsection{Existence of a popular difference}

We now turn to the context surrounding our popular difference results (Theorem \ref{thm:2D-pop-diff} and Theorem \ref{thm:1D-pop-diff}). To our knowledge, the first results in this direction concern three-term and four-term arithmetic progressions \cite{BHKMultiple}. These results were obtained via ergodic methods and as a result concern the weaker notion of {upper Banach density}, rather than the lower density we deal with in Theorem \ref{thm:2D-pop-diff} and Theorem \ref{thm:1D-pop-diff}. There have since been a number of ergodic works dealing with other configurations, see for instance \cite{frantzikinakis2008multiple, QingMultiple, DLMSOptimal}. In particular, Chu, Frantzikinakis and Host \cite{chu2011ergodic} have proved the ergodic (and hence qualitative) analogue of Theorem \ref{thm:2D-pop-diff}.

Green \cite{GreenSzemeredi} pioneered the use of arithmetic regularity to study these problems, proving the existence of a popular difference for three-term arithmetic progressions and thereby strengthening \cite{BHKMultiple} from upper Banach density to lower density. Green's theorem states that for any $\eps >0$ and any $A\subset [N]$ with $|A|\geq \delta N$, either $N \leq C(\eps)$ or there exists $d \neq 0$ such that
\begin{equation}\label{eq:green}
\hash\set{x\in A: x+d,\ x+2d \in A} \geq (\delta^3-\eps)N.
\end{equation}
Green and Tao \cite{GreenTaoArithmetic} used higher-order Fourier analysis to obtain the analogous result for four-term arithmetic progressions. The result is false for longer progressions; see Ruzsa's appendix to \cite{BHKMultiple}. 

In Green's argument for \eqref{eq:green}, the constant $C(\eps)$ is bounded in terms of a tower of twos of height $\eps^{-O(1)}$. This tower height was reduced to $O(\log(1/\eps))$ by Fox, Pham and Zhao \cite{FoxPhamPopularI, FoxPhamPopularII, FPZTower}, who went on to demonstrate (remarkably) that this bound is best possible. 

In higher dimensions, existence of a popular common difference for certain non-singular matrix progressions of length 3 and 4 was established in \cite{bergeretal2022popular}.  For corners \eqref{eq:corners}, the ergodic version of a popular difference was studied by Chu \cite{QingMultiple}, with combinatorial analogues obtained by Mandache \cite{mandache2021variant}. Both of these authors establish that the naive conjecture is not correct for the popular difference version of the corners theorem. The ``correct'' conjecture was emphatically addressed in \cite{foxetal2020triforce}. In \cite{sah2021patterns}, linear configurations were classified according to  when the naive conjecture is correct.

All of the above results concerning popular differences have either been qualitative or
the bounds obtained have been tower-type. The only result we are aware of with reasonable
bounds is due to Lyall and Magyar \cite{lyall2013optimal}, and concerns two-point
polynomial progressions in one dimension. Lyall and Magyar show that if $P$ is a
polynomial with integer coefficients and zero constant term, then for any $\eps >0$ and
$A\subset [N]$ with $|A|\geq \delta N$, either there exists $d \in \Z\setminus\set{0}$
such that
$$
\hash\set{x\in A: x+P(d) \in A} \geq (\delta^2-\eps)N,
$$ 
or $N \leq \exp\exp\brac{O\brac{\eps^{-1}\log(\eps^{-1}}}$. We believe that by suitably adapting the methods of this paper, one should be able to replace this double exponential bound with $N \leq \exp\brac{\eps^{-O(1)}}$. Determining the correct bound on $N$ seems an interesting problem, even in the simplest situation of $x$, $x+d^2$. In forthcoming work by the second author, the third author and Mengdi Wang, we generalize Lyall and Magyar's result to certain multi-point polynomial progressions, proving that if $P_1,\cdots,P_m$ are polynomials with integer coefficients and zero constant terms with distinct degrees, then for any $\eps > 0$ and
$A\subset [N]$ with $|A|\geq \delta N$, either there exists $d \in \Z\setminus\set{0}$
such that
$$
\hash\set{x\in A: x+P_1(d),\cdots,x+P_m(d) \in A} \geq (\delta^{m+1}-\eps)N,
$$ 
or $N \leq \exp\exp\brac{O\brac{\eps^{-1}}}$.

\section{An outline of our argument}\label{sec:outline}
All statements in this section are of a heuristic and informal nature, in particular much of what we write is, strictly speaking, false. To avoid obfuscating our sketch by tracking ranges of summation, we simply write $\E_x f(x)$ to indicate a normalised sum, with the range of summation to be inferred from the context.

Given a set $A \subset [N]^2$ of density $\delta := |A|/N^2$, our starting point is to study the count
\begin{equation}\label{count1}
\E_{x,y,d} 1_A(x,y)1_A(x+d,y) 1_A(x,y+d^2).
\end{equation}
Traditionally in additive combinatorics, one compares this with the count of configurations lying in a random set of density $\delta$, which (after normalisation) is of order $\delta^3$. If these counts are close, then one has a profusion of configurations in the set $A$. If not, then we hope to show that $A$ has some exploitable structure. In a density increment argument, we typically exploit this by passing to a substructure where the density is larger. Iterating this argument eventually yields a substructure where the count of configurations \textit{is} close to the random count, finishing our proof.

For many multidimensional configurations, we contend that the random count $\delta^3$ is
not a useful comparison to make. One reason for this is that there are very unstructured
sets that do not have the random number of configurations: taking an example from
\cite[\S5]{GreenFinite}, consider a product of random sets $B_1 \times B_2$ where each
$B_i\subset [N]$ has density $\sqrt{\delta}$. In the corners problem \eqref{eq:corners},
one is forced to deal with these very loosely structured objects - they only possess the
structure of a Cartesian product. However, the inverse theorem for our configuration
(Theorem \ref{thm:inverse}) says that the sets we need to deal with have more structure:
they take the form $B\times P$ where $P\subset [N]$ is a \textit{dense} arithmetic
progression (i.e., of length proportional to $N$ and with small common difference) and $B\subset [N]$ is arbitrary. After moving to a grid with slightly larger step size, this purported ``bad'' set can be thought of as taking the form $B'\times [N']$, where $B'$ has no exploitable structure.

With the above discussion in mind, and following \cite{han2021polynomial, kuca2024multidimensional}, we contend that it is more sensible to compare \eqref{count1} with the following count, where we have essentially replaced the third indicator function $1_A(x,y+d^2)$ with its average $\E_{y'} 1_A(x,y')$ on vertical fibres:
\begin{multline}\label{count2}
\E_{x,y,d} 1_A(x,y)1_A(x+d,y) \E_{d'}1_A(x,y+  d'+d^2)\\
\approx \E_{x,y,d,d'} 1_A(x,y)1_A(x+d,y) 1_A(x,y + d')\\
\approx \E_{x,y,x',y'} 1_A(x,y)1_A(x',y) 1_A(x,y').
\end{multline}
 A (non-obvious\footnote{See \url{https://mathoverflow.net/questions/189222}.}) application of H\"older's inequality gives that
$$
\E_{x,y,x',y'} 1_A(x,y)1_A(x',y) 1_A(x,y') \geq \brac{\E_{x,y} 1_A(x,y)}^3.
$$
 Hence, we have a profusion of configurations, and, in particular, the existence of a popular difference, if \eqref{count1} is $\eps$-close to \eqref{count2}.
 
 If \eqref{count1} is not $\eps$-close to \eqref{count2}, then, on setting
\begin{equation}\label{eq:f2-def}
f_2(x,y) := 1_A(x,y) - \E_{d'} 1_A(x,y+d'),
\end{equation}
we have
$$
\abs{\E_{x,y,d} 1_A(x,y)1_A(x+d,y) f_2(x,y+d^2)} \geq \eps.
$$
At this point, we apply our inverse theorem (Theorem \ref{thm:inverse}) to deduce the existence of $\alpha \approx a/q$ with $q\leq \eps^{-O(1)}$ such that
$$
\E_x\abs{\E_y f_2(x,y)e( \alpha y)} \geq \eps^{O(1)}.
$$
The function $y \mapsto e(\alpha y)$ is approximately invariant under small shifts of the form $y \mapsto y + q^2d'$, so by a little averaging, there exists $N' \geq \eps^{O(1)} N$ such that
\begin{equation}\label{eq:large-convolve}
\E_{x,y}\abs{\E_{|d'|<N'} f_2(x,y+q^2 d')} \geq \eps^{O(1)}.
\end{equation}
Equation \eqref{eq:large-convolve} is saying that we can convolve $f_2$ in the second coordinate with a long arithmetic progression of small common difference, in order to yield a function with large $L^1$-norm.  Recall from \eqref{eq:f2-def} that $f_2$ is, essentially, $1_A$ minus the convolution of $1_A$ in the second coordinate with the interval $(-N,N)$. We can, therefore, use \eqref{eq:large-convolve}, together with approximate orthogonality properties of convolution, in order to obtain an \textit{energy increment}:
$$
\E_{x,y}\abs{\E_{|d'|<N'} 1_A(x,y+q^2 d')}^2 \geq \E_{x,y}\abs{\E_{|d|<N} 1_A(x,y+ d)}^2 + \eps^{O(1)}.
$$

We now repeat the above argument, this time focusing on the \textit{restricted} counting operator
\begin{equation*}
\E_{x,y}\E_{|d|<N'} 1_A(x,y)1_A(x+qd,y) 1_A(x,y+q^2d^2),
\end{equation*}
which counts configurations where the difference parameter is divisible by $q$ and ranges over a shorter interval. We compare this restricted counting operator with the restricted analogue of \eqref{count2}, namely
$$
\E_{x,y}\E_{|d|,|d'|< N'} 1_A(x,y)1_A(x+qd,y) 1_A(x,y+q^2d').
$$
Again, this comparison either yields a profusion of configurations in our set $A$, or alternatively we have an energy increment of the form
$$
\E_{x,y}\Bigabs{\E_{| \tilde d|<\tilde N} 1_A(x,y+q^2\tilde{q}^2  \tilde d)}^2 \geq \E_{x,y}\Bigabs{\E_{|d'|<N'} 1_A(x,y+ q^2d')}^2 + \eps^{O(1)},
$$
for some $\tilde q \ll \eps^{-O(1)}$ and some $\tilde N \gg \eps^{O(1)} N'$.

Iterating the energy increment $n$ times, we obtain  a convolution whose energy is at least $n\eps^{O(1)}$. Since the energy is bounded above by 1, the energy increment must terminate after at most $n \leq \eps^{-O(1)}$ steps. Setting $q_1:= q$, $q_2:= q\tilde q$ and so on, the iteration yields some common difference $q_n$ of size
$$
q_n \leq \underbrace{\eps^{-O(1)}\dotsm \eps^{-O(1)}}_{\eps^{-O(1)} \text{ times}} = \exp\brac{\eps^{-O(1)}} 
$$
and an interval length $N_n$ satisfying
$$
N_n \geq \underbrace{\eps^{O(1)}\dotsm \eps^{O(1)}}_{\eps^{-O(1)} \text{ times}} N = N/\exp\brac{\eps^{-O(1)}},
$$ such that we have the comparison
\begin{multline*}
|\E_{x,y}\E_{|d|< N_n} 1_A(x,y)1_A(x+q_nd,y) 1_A(x,y+q_n^2d^2)\\
- \E_{x,y}\E_{|d|,|d'|< N_n} 1_A(x,y)1_A(x+q_nd,y) 1_A(x,y+q_n^2d')| \leq \eps.
\end{multline*}
This delivers a popular difference, since an argument using H\"older's inequality again gives something of the form
$$
\E_{x,y}\E_{|d|,|d'|< N_n} 1_A(x,y)1_A(x+q_nd,y) 1_A(x,y+q_n^2d') \gtrapprox \delta^3.
$$
\subsection{Paper organisation}
Proving our inverse theorem (Theorem \ref{thm:inverse}) occupies
\S\S\ref{sec:pet}-\ref{sec:inv-thm}. The first part of this argument (\S\ref{sec:pet})
shows that our configuration is controlled by a Gowers $U^5$-norm in the vertical
direction. This uses the PET induction scheme of Bergelson-Leibman
\cite{BergelsonLeibmanPolynomial} together with the ``quantitative concatenation''
machinery developed in \cite{PelusePrendivilleQuantitative}. In \S\ref{sec:deg-lower}, we
use the degree lowering technique as developed in \cite{PelusePrendivilleQuantitative} to reduce the $U^5$-norm to a $U^1$-norm (after passing to a long progression with small common difference). In \S\ref{sec:inv-thm}, we combine the results of the previous  sections in order to derive our inverse theorem. Our energy increment argument is carried out in \S\ref{sec:energy} and we deduce the existence of a popular common difference in \S\ref{sec:pop-com-diff}. Notation that is used repeatedly is collected in Appendix \ref{sec:notation}.
\section{PET induction and $U^5$-control in the vertical direction}\label{sec:pet}
%
We begin by showing that our counting operator is controlled by a Gowers $U^5$-norm in the vertical direction. In order to establish the existence of a popular common difference (Theorem \ref{thm:2D-pop-diff}), it is  convenient to work with a counting operator that incorporates a nicer weight than the sharp cut-off $d \in [-N,N]$ used in the inverse theorem stated in our introduction (Theorem \ref{thm:inverse}). 
\begin{theorem}[$U^5$-control in the vertical direction]\label{global U5}
Let $f_0, f_1, g : \Z^2 \to \C$ be 1-bounded functions and let $\mu_N$ denote the probability measure\footnote{We normalise convolution on $\Z$ with counting measure, see \eqref{convolution}. We write $x_+ := \max\set{x, 0}$ for the positive part of a real number.}
$$
\mu_N := \trecip{N^2}1_{[N]}*1_{-[N]}= \trecip{N}\brac{1-\tfrac{|\cdot|}{N}}_+.
$$
Write
$$
\Lambda_N(f_0, f_1, g) := \E_{x \in [N]}\E_{y \in [N^2]} f_0(x,y)\sum_d \mu_N(d)f_1(x+d,y) g(x,y+d^2).
$$
If $\abs{\Lambda_N(f_0,f_1, g)} \geq \delta$, then
$$
\E_{x \in [N]}\norm{g_{x}1_{[2N^2]}}_{U^5(\Z)}^{2^5} \gg \delta^{O(1)} \norm{1_{[2N^2]}}_{U^5(\Z)}^{2^5},
$$
where $g_{x}(y):=g(x,y)$ and the $U^5$-norm is defined via \eqref{Us def}.
\end{theorem}
\begin{proof}
Rather than summing $x$ and $y$ over intervals, it is convenient to be able to sum over $\Z$, as this allows us to change variables in a clean manner, without having to keep track of how this affects the range of summation. We may therefore assume that 
$$
\supp(f_0) \subset [N]\times [N^2],\quad \supp(f_1) \subset [-N,2N]\times [N^2],\quad \supp(g) \subset [N]\times [2N^2].
$$
Notice that this affects neither the hypothesis nor the conclusion of the theorem.

Following the procedure of \cite[Lemma 3.3]{PelusePrendivilleQuantitative}, we repeatedly
apply the Cauchy--Schwarz inequality and a change of variables. It is convenient to renormalise by writing $\tilde \mu_N:= N\mu_N$, to give a 1-bounded function with support contained in $[-N,N]$. 

Averaging the $d$ parameter by shifts $d \mapsto d+a$ with $a \in [N]$, and taking into account the support of $\tilde\mu_N$, we have that
$$
\Lambda_N(f_0, f_1, g) = \recip{N^4} \sum_{x,y} f_0(x,y) \sum_{|d|\leq N}\E_{a\in [N]}\tilde \mu_N(d+a) f_1(x+d+a, y) g(x, y+(d+a)^2).
$$
We apply the Cauchy--Schwarz inequality to double the $a$ variable in the above sum. After a further change of variables, recalling that $\Delta_{(a,0)} f_1(x,y) = f_1(x,y)\overline{f_1(x+a,y)}$ as in \eqref{eq:diff} and $\mu_N(a) := N^{-2}\sum_{a_1-a_2 = a} 1_{[N]}(a_1)1_{[N]}(a_2)$ as in \eqref{fejer}, this gives that
\begin{multline*}
\abs{\Lambda_N(f_0, f_1, g) } \leq
\brac{ \recip{N^4} \sum_{x,y} |f_0(x,y)|^2 \sum_{|d|\leq N} 1}^{1/2}\\
\brac{ \recip{N^4}\sum_{x,y,d, a}\mu_N(a)\Delta_a\tilde \mu_N(d)\Delta_{(a,0)}f_1(x, y) g(x-d, y+d^2)\overline{ g(x-d, y+(d+a)^2)}}^{1/2}.
\end{multline*}

We now repeat the above process, averaging $d$ by shifts $d \mapsto d+b$ with $b \in [N]$, then applying Cauchy--Schwarz to double the $b$ variable. After further change of variables, this gives  
\begin{multline*}
\abs{\Lambda_N(f_0, f_1, g) }^4 \ll  \recip{N^4}
\sum_{a,b,x,y,d}\mu_{N}(a)\mu_{N}(b) \Delta_{a,b}\tilde\mu_N(d) g(x,y)\overline{g\bigbrac{x,y+2ad + a^2}}\\ \overline{g\bigbrac{x-b,y+2bd + b^2}}g\bigbrac{x-b,y+2(a+b)d + (a+b)^2}.
\end{multline*}

We continue to iterate this procedure, this time averaging $d$ by shifts $d \mapsto d+h_1$ with $h_1\in [H]$, where $1\leq H \leq N$. Doing so gives
\begin{multline*}
\abs{\Lambda_N(f_0, f_1, g) }^8 \ll 
 \recip{N^4}\sum_{a,b, h_1,x,y,d}\mu_{N}(a,b)\mu_{H}(h_1) \Delta_{a,b,h_1}\tilde\mu_N(d) \overline{\Delta_{(0,2ah_1)}g\bigbrac{x,y+ a^2}}\\ \overline{\Delta_{(0,2bh_1)}g\bigbrac{x-b,y+2(b-a)d + b^2}}\Delta_{(0,2(a+b)h_1)}g\bigbrac{x-b,y+2bd + (a+b)^2},
\end{multline*}
using the notation in~\eqref{multidim fejer}. Another iteration yields
\begin{multline*}
\abs{\Lambda_N(f_0, f_1, g) }^{16} \ll  
\recip{N^4}\sum_{a,b, h_1,h_2,x,y,d}\mu_{N}(a,b)\mu_{H}(h_1,h_2) \Delta_{a,b,h_1,h_2}\tilde \mu_N(d)\\
 \overline{\Delta_{2bh_1,2(b-a)h_2}g_x\bigbrac{y + b^2}}
\Delta_{2(a+b)h_1, 2bh_2}g_x\bigbrac{y+2ad+(a+b)^2}.
\end{multline*}
Introducing an additional average over $h_3 \in [N]$ and changing variables, we have
\begin{multline*}
\abs{\Lambda_N(f_0, f_1, g) }^{16} \ll  
 \recip{N^4}\sum_{a,b, \vh,x,y,d}\mu_{N}(a,b)\mu_{H}(\vh) \E_{h_3\in[H]}\Delta_{a,b,\vh}\tilde \mu_N(d+h_3)\\
 \overline{\Delta_{2bh_1,2(b-a)h_2}g_x\bigbrac{y-2ad + b^2}}
\Delta_{2(a+b)h_1, 2bh_2}g_x\bigbrac{y+2ah_3+(a+b)^2}
\end{multline*}

We next observe that, for $1 \leq h_3\leq H \leq N$, we have
$$
\Delta_{a,b,\vh}\tilde \mu_N(d+h_3)= \Delta_{a,b,\vh}\tilde \mu_N(d) +O\brac{1_{[-2N, 2N]}(d)H/N}.
$$
Hence,
\begin{multline*}
\abs{\Lambda_N(f_0, f_1, g) }^{16} \ll  \frac{H}{N}+
 \recip{N^4}\biggl|\sum_{a,b, \vh,x,y,d}\mu_{N}(a,b)\mu_{H}(\vh) \sum_d\Delta_{a,b,\vh}\tilde \mu_N(d)\\
 \overline{\Delta_{2bh_1,2(b-a)h_2}g_x\bigbrac{y-2ad + b^2}}
\E_{h_3\in[H]}\Delta_{2(a+b)h_1, 2bh_2}g_x\bigbrac{y+2ah_3+(a+b)^2}\biggr| .
\end{multline*}
So, a final application of the Cauchy--Schwarz inequality gives
\begin{equation*}
\abs{\Lambda_N(f_0, f_1, g) }^{32} \ll  
\frac{H^2}{N^2}+\abs{
 \recip{N^3}\sum_{a,b, \vh,x,y}\mu_{N}(a,b)\mu_{H}(\vh)\Delta_{2(a+b)h_1, 2bh_2,2ah_3}g_x\bigbrac{y}} .
\end{equation*}

It follows that if $|\Lambda_N(f_0, f_1, g)| \geq \delta$, then either $H \gg \delta^{16} N$, or there exists a set $S \subset [N]$ of size $|S| \gg \delta^{32} N$ such that, for each $x \in S$, we have
$$
 \sum_{a,b \in (-2N, 2N)}\abs{\sum_h\mu_H(h)\sum_{y} \Delta_{(a+b)h_1,\, bh_2,\, ah_3}g_{x}(y)} \gg \delta^{32} N^4.
$$
For each such $x \in S$, the above gives the same conclusion as the first displayed equation in the proof of \cite[Theorem 5.6]{PelusePrendivilleQuantitative}. We note that, in the notation of \cite[Theorem 5.6]{PelusePrendivilleQuantitative}, we are replacing $N$ with $2N^2$, and taking $M := N$, $q:=1$ and $f:=g_{x}$ (which is a 1-bounded function with support contained in $[2N^2]$). Applying the proof of \cite[Theorem 5.6]{PelusePrendivilleQuantitative}, for each $x \in S$ we deduce that either $N \ll 1$ or
$$
\norm{g_{x}}_{U^5}^{32} \gg \delta^{O(1)} \norm{1_{[2N^2]}}_{U^5}^{32}.
$$
The conclusion of our theorem being trivial in the case that $N \ll 1$, the result follows on summing over $x \in S$.
\end{proof}

As is exploited repeatedly in \cite{PelusePrendivilleQuantitative, PelusePrendivillePolylogarithmic, PeluseBounds}, it is often more convenient to replace arbitrary bounded functions with dual functions, as these functions possess more structure. The next result shows that this is possible in Theorem \ref{global U5}. We eventually use this in \S\ref{sec:inv-thm}.
\begin{corollary}[Vertical $U^5$-control of the dual]\label{cor:dual-U5}
Let $f_0, f_1, g : \Z^2 \to \C$ be 1-bounded functions and let $\mu_N$ denote the probability measure
$$
\mu_N := \trecip{N^2}1_{[N]}*1_{-[N]}= \trecip{N}\brac{1-\tfrac{|\cdot|}{N}}_+.
$$ 
Write
$$
\Lambda_N(f_0, f_1, g) := \E_{x \in [N]}\E_{y \in [N^2]}\sum_d \mu_N(d) f_0(x,y)f_1(x+d,y) g(x,y+d^2).
$$
If $\abs{\Lambda_N(f_0,f_1, g)} \geq \delta$ then, on defining the dual 
\begin{equation}\label{dual defn}
F_{x}(y)=F(x,y) := \sum_d \mu_N(d) f_0(x,y-d^2) f_1(x+d,y-d^2),
\end{equation}
we have 
$$
\E_{x \in [N]}\norm{F_{x}1_{[2N^2]}}_{U^5(\Z)}^{2^5} \gg \delta^{O(1)} \norm{1_{[2N^2]}}_{U^5(\Z)}^{2^5},
$$
where the $U^5$-norm is defined via \eqref{Us def}.
\end{corollary}

\begin{proof}
We may assume that 
$$
\supp(f_0) \subset [N]\times [N^2],\quad \supp(f_1) \subset [-2N,2N]\times [N^2],\quad \supp(g) \subset [N]\times [2N^2],
$$
so that $\supp(F) \subset [N]\times [2N^2]$. With these assumptions on the support, we may remove the restriction that $(x,y) \in [N]\times[N^2]$ in our counting operator, instead summing over $\Z^2$. Thus
$$
\delta  \leq |\Lambda_N(f_0, f_1, g)| = \recip{N^3}\abs{ \sum_{x,y} F(x,y) g(x,y)}.
$$
The Cauchy--Schwarz inequality then gives
$$
\delta^2 \ll \recip{N^3}\sum_{x,y}  \overline{F(x,y)}F(x,y) = \Lambda_N\brac{\overline{f}_0, \overline{f}_1,F} .
$$
By Theorem \ref{global U5}, we deduce that
$$
\E_{x \in [N]}\norm{F_{x}1_{[2N^2]}}_{U^5(\Z)}^{2^5} \gg \delta^{O(1)} \norm{1_{[2N^2]}}_{U^5(\Z)}^{2^5} .
$$
\end{proof}

\section{Degree lowering}\label{sec:deg-lower}
The key result of this section is Lemma \ref{lem:deg-lower}, which shows that we can replace the $U^5$-norm in Corollary \ref{cor:dual-U5} with the $U^1$-norm, after passing to a long subprogression of small common difference. Before proving this, we begin with a number of technical lemmas that underlie the argument.
\begin{lemma}[Weyl bound]\label{lem:weyl-bound}
Let $\mu_N$ denote the probability measure
$$
\mu_N := \trecip{N^2}1_{[N]}*1_{-[N]}= \trecip{N}\brac{1-\tfrac{|\cdot|}{N}}_+.
$$
For any $\alpha, \beta \in \T$ we have the implication
\begin{multline}\label{eq:weyl-type-2.5}
\abs{\sum_{d} \mu_N(d) e(\alpha d^2 + \beta d)} \geq \delta\\ \implies \brac{\exists q \ll \delta^{-O(1)}}\sqbrac{ \norm{q\alpha}_{\T}\ll \delta^{-O(1)}/N^2 \text{ and } \norm{q\beta}_{\T}\ll \delta^{-O(1)}/N}.
\end{multline}
Furthermore, for any integers $a,b \in \Z$ we have
\begin{equation}\label{eq:weyl-type-3}
\abs{\sum_{d} \mu_N(d+a)\mu_N(d+b) e(\alpha d^2 + \beta d)} \geq \delta/N \implies \brac{\exists q \ll \delta^{-O(1)}}\sqbrac{ \norm{q\alpha}_{\T}\ll \delta^{-O(1)}/N^2}.
\end{equation}
\end{lemma}
\begin{proof}
Let us prove \eqref{eq:weyl-type-3}, the remaining assertions being similar and simpler.

Expanding the definition of the Fej\'er kernel \eqref{fejer}, we have 
\begin{multline*}
\sum_{d} \mu_N(d+a)\mu_N(d+b) e(\alpha d^2 + \beta d)\\ = N^{-4} \sum_{d_1, d_2, d_3} 1_{[N]}(d_1)1_{[N]}(d_2)1_{[N]}(d_3)1_{[N]}(a-b-d_1+d_2+d_3)\\ e\brac{\alpha\brac{d_1-d_2-a}^2 + \beta\brac{d_1-d_2-a}}\\
\leq N^{-2}\max_{c_1,c_2}\abs{\sum_{d_1} 1_{[N]}(d_1)1_{[N]}(c_1 -d_1)e\brac{\alpha d_1^2 + (\beta-c_2\alpha)d_1}} .
\end{multline*}
Since the final summation in $d_1$ is over a  subinterval of $[N]$, the claimed conclusion follows  from Weyl's inequality, as formulated in \cite[Lemma 1.1.16]{TaoHigher}.
\end{proof}
\begin{lemma}[Shifted quadratic correlations are major arc]\label{lem:dual-major}
Let $f : \Z \to \C$ be a 1-bounded function and let $\mu_N$ denote the probability measure
$$
\mu_N := \trecip{N^2}1_{[N]}*1_{-[N]}= \trecip{N}\brac{1-\tfrac{|\cdot|}{N}}_+.
$$
Suppose that for some $\alpha \in \T$ we have
\begin{equation}\label{eq:lin-shift}
\E_{x\in [N]}\abs{\sum_{ d}\mu_N(d) f(x+d) e\brac{\alpha  d^2}} \geq \delta .
\end{equation}
Then there exists $q \ll \delta^{-O(1)}$ such that $\norm{q\alpha}_{\T} \ll \delta^{-O(1)}/N^2$. Furthermore there exists $\beta \in \T$ such that $\norm{q\beta}_{\T} \ll \delta^{-O(1)}/N$ and
$$
\abs{\sum_{ x \in[-N,2N]} f(x)e(\beta x)}\gg\delta^2 N.
$$

Similarly, if for some $\beta \in \T$ we have
\begin{equation}\label{eq:quad-shift}
\E_{y\in [N^2]}\abs{\sum_{ d}\mu_N(d) f(y+d^2) e\brac{\beta  d}} \geq \delta, 
\end{equation}
then there exists $q \ll \delta^{-O(1)}$ such that
$\norm{q\beta}_{\T} \ll \delta^{-O(1)}/N$ and there exists $\alpha \in \T$ with
$\norm{q\alpha}_{\T} \ll \delta^{-O(1)}/N^2$ such that
$$
\abs{\sum_{ y \in[2N^2]} f(y)e(\alpha y)}\gg\delta^{O(1)} N^2.
$$
\end{lemma}
\begin{proof}
  Write $h(d) := \mu_N(d) e(\alpha d^2)$. Our assumption \eqref{eq:lin-shift} implies that
  there exists a phase function $g : \Z \to \set{z \in \C : |z| =1}$ such that
$$
\sum_{x,d} g(x)1_{[N]}(x)f(x+d)1_{[-N,2N]}(x+d) h(d) \geq \delta N.
$$
By orthogonality and H\"older's inequality, we have
\begin{multline*}
\delta N \leq \int_\T \widehat{g1}_{[N]}(\beta)\widehat{f1}_{[-N,2N]}(-\beta) \hat{h}( \beta)\intd \beta\\
 \leq \bignorm{\widehat{g1}_{[N]}}_2\bignorm{\widehat{f1}_{[-N,2N]}}_2^{1/2}\bignorm{\hat{h}}_2^{1/2}\bignorm{\widehat{f1}_{[-N,2N]}\hat{h}}_\infty^{1/2}\\
 \leq\norm{1_{[N]}}_2\norm{1_{[-N,2N]}}_2^{1/2}\norm{\mu_N}_2^{1/2}\bignorm{\widehat{f1}_{[-N,2N]}\hat{h}}_\infty^{1/2} .
\end{multline*}
Hence there exists $\beta \in \T$ such that
$$
\biggabs{\sum_{d}\mu_N(d)e(\alpha d^2 + \beta d)} \gg \delta^2 \quad\text{and}\quad \abs{\sum_{ x \in[-N,2N]} f(x)e(\beta x)}\gg\delta^2 N.
$$
The result now follows from the Weyl bound (Lemma \ref{lem:weyl-bound}).

Next we turn to \eqref{eq:quad-shift}. This implies that there exists a phase function $g : \Z \to \set{z \in \C : |z| =1}$ such that
$$
\sum_{y,d} g(y)1_{[N^2]}(y)f(y+d^2)1_{[2N^2]}(y+d^2) \mu_N(d)e(\beta d) \geq \delta N^2.
$$
Write $S_N(\alpha, \beta) := \sum_d \mu_N(d) e(\alpha d^2+\beta d)$. By orthogonality and H\"older's inequality, we have
\begin{multline*}
\delta N^2 \leq \int_\T \widehat{g1}_{[N^2]}(\alpha)\widehat{f1}_{[2N^2]}(-\alpha)S_N(\alpha, \beta)\intd \alpha\\
 \leq \bignorm{\widehat{g1}_{[N^2]}}_2\bignorm{\widehat{f1}_{[2N^2]}}_2^{5/7}\norm{S(\cdot,\beta)}_6^{6/7}\bignorm{\bigabs{\widehat{f1}_{[2N^2]}}^{2/7}\abs{S(\cdot, \beta)}^{1/7}}_\infty.
\end{multline*}

Employing a standard estimate for the sixth moment of a quadratic exponential sum (see, for instance, \cite[Lemma 6.4]{PelusePrendivillePolylogarithmic}), there exists $\alpha \in \T$ such that
$$
\biggabs{\sum_{d}\mu_N(d)e(\alpha d^2 + \beta d)} \gg \delta^7 \quad\text{and}\quad \abs{\sum_{ x \in[2N^2]} f(x)e(\alpha x)}\gg\delta^{7/2} N^2.
$$
The result now follows from the Weyl bound (Lemma \ref{lem:weyl-bound}).
\end{proof}
The following simple observation is used repeatedly in the remainder of the paper, and can be regarded as the key ``reason'' why Theorem \ref{thm:1D-pop-diff} has exponential bounds, rather than the tower-type bounds encountered in the popular version of the linear Roth theorem \cite{FPZTower}. It essentially allows us to show that we can pass from \textit{many} ``major arc'' correlations to  a \textit{single} correlation: the major arcs line up. The reason underlying this is that the major arcs form a tiny fraction of all possible frequencies. 
\begin{lemma}[Pigeon-holing major arcs]\label{lem:PH-major}
Let $f_1, \dots, f_M : [-N,N] \to \C$ be 1-bounded functions. For a fixed positive integer $k$, let $\alpha,\alpha_1, \dots, \alpha_M \in \T$ be frequencies for which there are positive integers  $q_1, \dots, q_M \leq Q$ such that $\norm{q_m(\alpha_m-\alpha)}_{\T} \leq Q/N^k$ for each $m$. Suppose that
$$
\E_{m\in [M]}\abs{\E_{|x|\leq N} f_m(x)e(\alpha_m x^k)} \geq 1/Q.
$$
Then there exists $m_0 \in [M]$ such that
$$
\E_{m\in [M]}\abs{\E_{|x|\leq N} f_{m}(x)e(\alpha_{m_0} x^k)} \gg 1/Q^{O(1)}.
$$
Moreover, if $\alpha_1, \dots, \alpha_M$ are $\recip{ QN^k}$-separated\footnote{By which we mean that, for each $i$ and $j$, either $\alpha_i = \alpha_j$ or $\norm{\alpha_i-\alpha_j}\geq \recip{QN^k}$.}, then we can ensure that the $m_0 \in [M]$ found above satisfies
$$
\hash\set{m \in [M] : \alpha_{m} = \alpha_{m_0}}  \gg MQ^{-O(1)}.
$$
\end{lemma}
\begin{proof}
By the popularity principle, there exists a set $\mathcal{M} \subset [M]$ with $|\mathcal{M}| \geq  \frac{M}{2Q}$  such that for each $m \in \mathcal{M}$ we have
\begin{equation}\label{eq:mult-freq}
\abs{\E_{|x|\leq N} f_m(x)e(\alpha_m x^k)} \geq \tfrac{1}{2Q}.
\end{equation}
Set $T := 100QN^k$. If $\alpha_1, \dots, \alpha_M$ are $1/(QN^k)$-separated, then set $\tilde \alpha_m := \alpha_m$; otherwise, we round each $\alpha_m$ to the nearest fraction of the form $t/T$. In either case, we obtain a $1/T$-separated set $\tilde \alpha_1, \dots, \tilde \alpha_M$ contained in the shifted union
$$
\alpha+ \bigcup_{1 \leq a \leq q \leq Q}   \sqbrac{\tfrac{a}{q}- \tfrac{2Q}{ qN^k},  \tfrac{a}{q}+\tfrac{2Q}{ qN^k}}.
$$
Write $D$ for the number of distinct $\tilde \alpha_m$ lying in the above set. Since the $\tilde \alpha_m$ are $1/T$-separated, a volume packing argument gives
$$
D \leq \sum_{q\leq Q}q(1+\tfrac{4QT}{qN^k})\ll Q^2\brac{1+\tfrac{T}{N^k}} \ll Q^3 .
$$  
Hence, by the pigeon-hole principle, there exists $\tilde\alpha_{m_0}$ such that
$$
\hash\set{m \in \mathcal{M} : \tilde\alpha_{m} = \tilde\alpha_{m_0}} \gg |\mathcal{M}|Q^{-3} \gg MQ^{-4}.
$$

Our choice of $T$ ensures that, for each $m$ with $\tilde \alpha_{m} = \tilde \alpha_{m_0}$, the frequencies $\alpha_{m}$ and $\alpha_{m_0}$ are sufficiently close to ensure that 
\begin{equation}\label{eq:mult-freq}
\abs{\E_{|x|\leq N} f_{m}(x)e(\alpha_{m_0} x^k)} \gg 1/Q.
\end{equation}
Summing over $m\in [M]$ with $\tilde \alpha_{m} = \tilde \alpha_{m_0}$, we deduce that 
$$
\E_{m\in [M]} \abs{\E_{|x|\leq N} f_{m}(x)e(\alpha_{m_0} x^k)} \gg 1/Q^5.
$$
\end{proof}
Our penultimate technical lemma before proving Lemma \ref{lem:deg-lower} is essentially the main argument underlying Lemma \ref{lem:deg-lower}, but carried out in a notationally simpler context.
\begin{lemma}[Shifted quadratic correlations conspire]\label{lem:duals-conspire}
Let $f_1, \dots, f_M : \Z \to \C$ be 1-bounded functions and let $\mu_N$ denote the probability measure
$$
\mu_N := \trecip{N^2}1_{[N]}*1_{-[N]}= \trecip{N}\brac{1-\tfrac{|\cdot|}{N}}_+.
$$
Let $T$ be a positive integer satisfying $T \leq \delta^{-1} N^2$. Suppose that there are functions $\phi_1, \dots, \phi_M : \Z \to \set{t/T : t \in [T]}$ and a set $\mathcal{N} \subset [M]\times [N]$ of size at least $ \delta MN$ such that for each $(m,x) \in \mathcal{N}$ we have
\begin{equation}\label{eq:large-dual-fourier}
\abs{\sum_{ d} \mu_N(d) f_m(x+d) e\brac{\phi_m(x)  d^2}} \geq \delta  .
\end{equation}
Then there exists $q \ll \delta^{-O(1)}$ and $\phi \in \T$ with $\norm{q\phi}_{\T}\ll \delta^{-O(1)}/N^2$ such that 
$$
\hash\set{(m,x) \in \mathcal{N} : \phi_m(x) = \phi} \gg \delta^{O(1)} M N.
$$
\end{lemma}
\begin{proof}
We first show that, for fixed $m$, many of the $\phi_m(x)$ lie in the same shift of the major arcs. 

By the popularity principle, there exists a set $\mathcal{M} \subset [M]$ with $|\mathcal{M}| \gg \delta M$  such that for each $m \in \mathcal{M}$ the fibre $\mathcal{N}_m := \set{x \in [N] : (m, x) \in \mathcal{N}}$ satisfies $|\mathcal{N}_m|\gg \delta N$. Summing over $\mathcal{N}_m$, we deduce that
$$
\E_{x \in \mathcal{N}_m}\abs{\sum_{ d} \mu_N(d)f_m(x+d) e\brac{\phi_m(x)  d^2}} \geq \delta  .
$$
Applying the Cauchy--Schwarz inequality to double the $d$ variable and then
re-parametrising gives
\begin{multline*}
\delta^2   \leq \E_{x \in \mathcal{N}_m}\sum_{ d, \tilde d } \mu_N(d)\mu_N(\tilde d) f_m(x+d)\overline{f_m(x+\tilde d)} e\brac{\phi_m(x)  \sqbrac{d^2-\tilde d^2}} \\
\ll \recip{\delta N}\sum_{|x|<2N}\sum_{|h| < 2N}f_m(x)\overline{f_m(x+h)} \sum_{d}\Delta_h\mu_N(d) 1_{\mathcal{N}_m}(x-d)e\brac{-\phi_m(x-d)  \sqbrac{2hd+h^2}}.
\end{multline*}

We again apply the Cauchy--Schwarz inequality to double the $d$ variable, yielding
\begin{align*}
\delta^6 & \ll 
\sum_{x,h,d,\tilde d} \Delta_h\mu_N(d)\Delta_h\mu_N(\tilde d)1_{\mathcal{N}_m}\bigbrac{x-d}1_{\mathcal{N}_m}\bigbrac{x-\tilde d\, }\\
&\qquad \times e\Bigbrac{\sqbrac{\phi_m(x-d)-\phi_m\bigbrac{x-\tilde d\, }}h^2+\phi_m(x-d)2hd-\phi_m\bigbrac{x-\tilde d\, }2h\tilde d}\\
&\leq 
\sum_{x,d,\tilde d} \mu_N(d)\mu_N(\tilde d)1_{\mathcal{N}_m}\bigbrac{x-d}1_{\mathcal{N}_m}\bigbrac{x-\tilde d\, }\biggl|\sum_h \mu_N(d+h)\mu_N(\tilde d + h) \\
&\qquad \times e\Bigbrac{\sqbrac{\phi_m(x-d)-\phi_m\bigbrac{x-\tilde d\, }}h^2+\phi_m(x-d)2hd-\phi_m\bigbrac{x-\tilde d\, }2h\tilde d}\biggr|.
\end{align*}
Next, we change variables in the above sum, substituting $d = x-n$ and $\tilde d = x-\tilde n$. We also use the pointwise bound $\mu_N \leq 1/N$, together with the fact that the only values of $x$ that contribute to the above sum lie in the interval $[-2N, 2N]$. Thus taking a maximum over $x$ gives some $x_0$ such that
\begin{align*}
\delta^6 N & \ll  
\sum_{n,\tilde n} 1_{\mathcal{N}_m}(n)1_{\mathcal{N}_m}(\tilde n)\biggl|\sum_h \mu_N(x_0-n+h)\mu_N(x_0-\tilde n + h) \\
&\qquad \times e\Bigbrac{\sqbrac{\phi_m(n)-\phi_m(\tilde n)}h^2+\phi_m(n)2h(x_0-n)-\phi_m(\tilde n)2h(x_0-\tilde n)}\biggr|.
\end{align*}
Taking a maximum over $\tilde n$ and using the popularity principle in $n$, there exists $\mathcal{N}_m'\subset \mathcal{N}_m$ with $|\mathcal{N}_m'| \gg \delta^6 N$ and there exists an integer $y_0$ along with a frequency $\alpha_m$ such that for each  $n \in \mathcal{N}_m'$ there exists an integer $y_n$ and a frequency $\beta_m(n)\in \T$ satisfying
\begin{equation*}
\frac{\delta^6}{N} \ll \biggl|\sum_h \mu_N(y_n+h)\mu_N(y_{0} + h) 
 e\Bigbrac{\sqbrac{\phi_m(n)-\alpha_m}h^2+\beta_m(n)h}\biggr|.
\end{equation*}
By the Weyl bound (Lemma \ref{lem:weyl-bound}), for each $n \in \mathcal{N}_m'$ there exists $q_{m,n} \ll \delta^{-O(1)}$ with 
$$
\norm{q_{m,n}\sqbrac{\phi_m(n)-\alpha_m}}_{\T}\ll \delta^{-O(1)}/N^2.
$$
Whilst summing \eqref{eq:large-dual-fourier} over $n \in \mathcal{N}_m'$ gives
$$
\sum_{n \in [N]}\abs{\sum_{ d} \mu_N(d) f_m(n+d) e\brac{\phi_m(n)  d^2}}\gg \delta^{O(1)} N.
$$

Applying Lemma \ref{lem:PH-major}, for each $m \in \mathcal{M}$ there exists $n_m \in \mathcal{N}'_m$ such that on setting $\phi_m := \phi_m(n_m)$ we have
$$
\hash\set{n \in \mathcal{N}_m : \phi_m(n) = \phi_m} \gg \delta^{O(1)} N.
$$
In particular
$$
\E_{x \in [N]}\abs{\sum_{ d}\mu_N(d) f_m(x+d) e\brac{\phi_m  d^2}} \gg \delta^{O(1)} .
$$
Applying Lemma \ref{lem:dual-major}, there exists $q \ll \delta^{-O(1)}$ such that $\norm{q\phi_m}_{\T} \ll \delta^{-O(1)}/N^2$. Re-applying Lemma \ref{lem:PH-major}  gives some $\phi = \phi_{m_0}$ such that
$$
\hash\set{m\in \mathcal{M} : \phi_m = \phi} \gg \delta^{O(1)} M
$$
Hence
$$
\hash\set{(m, x)\in \mathcal{N} : \phi_m(x) = \phi} \geq \sum_{m \in \mathcal{M}: \phi_m = \phi} \hash\set{x \in \mathcal{N}_m : \phi_m(x) = \phi_m} \gg \delta^{O(1)} MN.
$$
\end{proof}
We are now in a position to prove our degree lowering lemma.
\begin{lemma}[Degree lowering]\label{lem:deg-lower}
Let $f_0, f_1 : \Z \to \C$ be 1-bounded functions with 
$$
\supp(f_0) \subset [N]\times [N^2] \quad \text{and} \quad \supp(f_1) \subset [-N, 2N] \times [N^2]
$$
Let $\mu_N$ denote the probability measure
$$
\mu_N := \trecip{N^2}1_{[N]}*1_{-[N]}= \trecip{N}\brac{1-\tfrac{|\cdot|}{N}}_+.
$$ 
Define the dual function 
$$
F_{x}(y) = F(x,y) := \sum_d\mu_N(d) f_0(x,y-d^2) f_1(x+d,y-d^2),
$$
a 1-bounded function supported on $[N]\times [2N^2]$.  Suppose that for some $s \geq 2$ we have
$$
\E_{x\in [N]}\norm{F_{x}}_{U^s(\Z)}^{2^s} \geq \delta \norm{1_{[2N^2]}}_{U^s(\Z)}^{2^s},
$$
where the $U^s$-norm is given by \eqref{Us def}. Then, either $N \ll_s \delta^{-O_s(1)}$, or there exists $\alpha \in \T$ and $q \ll_s \delta^{-O_s(1)}$ such that $\norm{q\alpha}_{\T} \ll_s\delta^{-O_s(1)}/N^2$ and 
\begin{equation}\label{eq:major-arc-correlation}
\E_{x\in [N]}\E_{y\in [N^2]}\abs{\sum_{d}\mu_N(d) f_1(x+d, y) e\brac{\alpha d^2}} \gg_s \delta^{O_s(1)} .
\end{equation}
\end{lemma}

\begin{proof}  
Let us first assume that $s\geq 3$. In this case, we claim that our assumptions imply that
\begin{equation}\label{eq:deg-lower-ineq}
\E_{x\in [N]}\norm{F_{x}}_{U^{s-1}(\Z)}^{2^{s-1}} \gg_s \delta^{O_s(1)} \norm{1_{[2N^2]}}_{U^{s-1}(\Z)}^{2^{s-1}}.
\end{equation}
Given this claim, we intend to iterate in order to arrive at the case in which $s = 2$, so that
\begin{equation}\label{eq:deg-lower-to-U2}
\E_{x\in [N]}\norm{F_{x}}_{U^{2}(\Z)}^{4} \gg_s \delta^{O_s(1)}\norm{1_{[2N^2]}}_{U^{2}(\Z)}^{4}.
\end{equation}

Returning to the proof of \eqref{eq:deg-lower-ineq}, since we may assume that $N \ll_s 1$ does not hold, we have 
$$
\norm{1_{[2N^2]}}_{U^s(\Z)}^{2^s} \gg_s N^{2s+2}.
$$ 
Expanding the definition of the $U^s$-norm \eqref{Us def}, and using the fact that $F_x$ is supported on $[2N^2]$,  gives
$$
\E_{|h_1|< 2N^2} \dotsm\E_{|h_{s-2}|< 2N^2}\E_{x\in [N]} \norm{\Delta_{h_1, \dots, h_{s-2}}F_{x}}_{U^2}^4 \gg_s \delta N^6.
$$

Applying the $U^2$-inverse theorem \cite[Lemma A.1]{PelusePrendivilleQuantitative}, there exists 
$$
\mathcal{N} \subset \left((-2N^2, 2N^2)^{s-2}\cap\mathbb{Z}^{s-2}\right)\times [N]
$$ of size $|\mathcal{N}| \gg_s \delta N^{2s-3}$ and a function $\phi : \Z^{s-1} \to \T$ such that for every $(\vh, x) \in \mathcal{N}$ we have
\begin{equation}\label{phi correlation}
\biggabs{\sum_y \Delta_{\vh} F_{x}(y) e\bigbrac{\phi(\vh;x)y} }\gg_s \delta N^2.
\end{equation}

Set $T := \ceil{C_s\delta^{-1}N^2}$, with $C_s$ an absolute constant taken sufficiently large to ensure that, on rounding $\phi(\vh;x)$ to the nearest fraction of the form $t/T$, the inequality \eqref{phi correlation} remains valid.  

Summing over $ \mathcal{N}$ and applying the dual-difference interchange inequality \cite[Lemma 6.3]{PelusePrendivilleQuantitative}, we deduce that 
\begin{multline*}
 \sum_{(\vh^{0},x), (\vh^1, x) \in \mathcal{N}} \biggl|\sum_{y,d}\mu_N(d) \Delta_{(0,\vh^0-\vh^1)} f_0(x,y-d^2)\\ \Delta_{(0,\vh^0-\vh^1)}f_1(x+d,y-d^2)
  e\bigbrac{\phi(\vh^0; \vh^1;x)y}\biggr| \gg_s \delta^{O_s(1)} N^{4s-5},
\end{multline*}
where
$$
\phi(\vh^0; \vh^1;x) := \sum_{\omega \in \set{0, 1}^s} (-1)^{|\omega|} \phi(\vh^{\omega};x)\qquad \text{and} \qquad
\vh^{\omega} := (h_1^{\omega_1}, \dots, h_s^{\omega_s}).
$$
Changing variables from $y$ to $y+d^2$ and taking  maxima, there exist $y$ and $\vh^1$ such that 
\begin{equation*}
\sum_{(\vh^{0},x) \in \mathcal{N}} \abs{\sum_{d} \mu_N(d)\Delta_{(0,\vh^0-\vh^1)}f_1(x+d,y)
  e\bigbrac{\phi(\vh^0; \vh^1;x)d^2}} \gg_s \delta^{O_s(1)} N^{2s-3}.
\end{equation*}
Thus, there exists $\mathcal{N}' \subset \mathcal{N}$ of size $\mathcal{N}'\gg_s \delta^{O_s(1)} N^{2s-3}$  such that for each $(\vh^0, x) \in \mathcal{N}'$ we have
 \begin{equation*}
 \abs{\sum_{d} \mu_N(d)\Delta_{(0,\vh^0-\vh^1)}f_1(x+d,y)
  e\bigbrac{\phi(\vh^0; \vh^1;x)d^2}} \gg_s \delta^{O_s(1)} .
\end{equation*}

By Lemma \ref{lem:duals-conspire}, there exists $\mathcal{N}''\subset \mathcal{N}'$ of size $\gg_s\delta^{O_s(1)} N^{2s-3}$ and $\phi \in \T$ such that for any $(\vh^0,x)\in \mathcal{N}''$ we have $\phi(\vh^0; \vh^1; x) = \phi$.  In particular, when restricted to the set $\mathcal{N}''$, the function $\phi$ satisfies
$$
\phi(\vh^0; x) =  \phi - \sum_{\omega\in \set{0,1}^{s-2}\setminus\set{0}} (-1)^{|\omega|} \phi(\vh^{\omega};x).
$$
For fixed $x$, the right-hand side of this identity is \emph{low rank} in $\vh^0$ according to the terminology preceding \cite[Lemma 6.4]{PelusePrendivilleQuantitative}.

Summing over $\vh$ such that $(\vh, x)\in \mathcal{N}''$ in \eqref{phi correlation}, we deduce the existence of  low-rank functions $\psi_{x}: \Z^{s-2} \to \T$ such that
\begin{equation}\label{eq:almost-conc}
\E_{x\in [N]}\E_{|h_1|< N^2} \dotsm\E_{|h_{s-2}|< N^2}\abs{\sum_{y} \Delta_{\vh} F_{x}(y) e\bigbrac{\psi_{x}(\vh)y} }\gg_s \delta^{O_s(1)}N^2 .
\end{equation}
Employing the inequality stated in \cite[Lemma 6.4]{PelusePrendivilleQuantitative} then gives
\begin{equation*}\label{eq:deg-lower-ineq-2}
\E_{x\in [N]}\norm{F_{x}}_{U^{s-1}}^{2^{s-1}} \gg_s \delta^{O_s(1)}N^{2s}.
\end{equation*}
This proves the claim \eqref{eq:deg-lower-ineq}. 

By iterating \eqref{eq:deg-lower-ineq}, our initial assumption of large $U^s$-norm yields a large $U^2$-norm, as in \eqref{eq:deg-lower-to-U2}. Moreover, we note that when running the argument with $s= 3$, the inequality \eqref{eq:almost-conc} becomes 
$$
\E_{x\in [N]}\abs{\E_{|y|\leq 2N^2}  F_{x}(y) e\bigbrac{\psi_x y} }\gg_s \delta^{O_s(1)},
$$
where for each $x$, our application of Lemma \ref{lem:duals-conspire} gives some $q_x\ll_s \delta^{-O_s(1)}$ with $\norm{q_x\psi_x}_{\T} \ll_s \delta^{-O_s(1)}/N^2$. Once more applying Lemma \ref{lem:PH-major}, we may assume that $\psi_x = \psi$ is constant in $x$, so that
$$
\E_{x\in [N]}\biggabs{\sum_y  F_{x}(y) e\bigbrac{\psi y} }\gg_s \delta^{O_s(1)}N^2.
$$
Expanding the dual function $F_{x}$, changing variables in $y$ and employing the triangle inequality to bring the sum over $y$ outside the absolute value, we derive \eqref{eq:major-arc-correlation}.
\end{proof}
\section{Inverse theorems for our counting operator}\label{sec:inv-thm}
Degree lowering (Lemma \ref{lem:deg-lower}) can be combined with vertical $U^5$-control of
the dual function (Corollary \ref{cor:dual-U5}), and a little Fourier analysis, in order to prove an inverse theorem for the function weighting the $(x+d, y)$ term in our counting operator.

\begin{theorem}[Counting operator inverse theorem, I]\label{thm:inv-thm-1}
Let $f_0, f_1, f_2 : \Z^2 \to \C$ be 1-bounded functions and let $\mu_N$ denote the probability measure
$$
\mu_N := \trecip{N^2}1_{[N]}*1_{-[N]}= \trecip{N}\brac{1-\tfrac{|\cdot|}{N}}_+.
$$ 
Write
$$
\Lambda_N(f_0, f_1, f_2) := \E_{x \in [N]}\E_{y \in [N^2]}\sum_d \mu_N(d) f_0(x,y)f_1(x+d,y) f_2(x,y+d^2).
$$
If $\abs{\Lambda_N(f_0,f_1, f_2)} \geq \delta$, then either $N \ll \delta^{-O(1)}$ or there exists $\beta \in \T$ and $q \ll \delta^{-O(1)}$ such that $\norm{q\beta}_{\T} \ll\delta^{-O(1)}/N$ and
\begin{equation}\label{eq:1st-inverse-eq}
\E_{y\in [N^2]} \abs{\E_{x\in [-N, 2N]} f_1(x, y) e\brac{\beta x}} \gg \delta^{O(1)} .
\end{equation}
\end{theorem}
\begin{proof}
In order  to avoid  tracking the range of summation, we  assume that 
$$
\supp(f_0) \subset [N]\times [N^2] \quad \text{and} \quad \supp(f_1) \subset [-N, 2N] \times [N^2]
$$
Notice that this affects neither the hypothesis nor the conclusion of the theorem.

By vertical $U^5$-control of the dual function (Corollary \ref{cor:dual-U5}), our assumptions give that  
$$
\E_{x \in [N]}\norm{F_{x}}_{U^5(\Z)}^{2^5} \gg \delta^{O(1)} \norm{1_{[2N^2]}}_{U^5(\Z)}^{2^5},
$$
where
$$
F_{x}(y)=F(x,y) := \sum_d\mu_N(d) f_0(x,y-d^2) f_1(x+d,y-d^2).
$$
By degree lowering (Lemma \ref{lem:deg-lower}), either $N \ll \delta^{-O(1)}$ or there exists $\alpha \in \T$ and $q \ll \delta^{-O(1)}$ with $\norm{q\alpha}_{\T} \ll \delta^{-O(1)}/N^2$ such that
$$
\E_{x\in [N]}\E_{y\in [N^2]}\abs{\sum_{d}\mu_N(d) f_1(x+d,y)e(\alpha d^2)} \gg \delta^{O(1)}.
$$

By the popularity principle and Lemma \ref{lem:dual-major}, for each $y \in [N^2]$ there exist $q_y \ll \delta^{-O(1)}$ and $\beta_y \in \T$ with $\norm{q_y\beta_y}_{\T} \ll \delta^{-O(1)}/N$ such that 
$$
\E_{y\in [N^2]}\abs{\E_{x\in [-N, 2N]} f_1(x,y) e(\beta_y x)} \gg \delta^{O(1)}.
$$
Pigeon-holing in the major arcs (Lemma \ref{lem:PH-major}), there exists $\beta \in \T$ and $q \ll \delta^{-O(1)}$ with $\norm{q\beta}_{\T} \ll \delta^{-O(1)}/N$ such that 
$$
\E_{y\in[N^2]} \abs{\E_{x\in [-N, 2N]} f_1(x,y) e(\beta x)} \gg \delta^{O(1)} .
$$
The result follows.
\end{proof}

We now use the inverse theorem for the function weighting the $(x+d, y)$ term to obtain an inverse theorem for the  function weighting the $(x, y+d^2)$ term.

\begin{theorem}[Counting operator inverse theorem, II]\label{thm:inv-thm-2}
Let $f_0, f_1, f_2 : \Z^2 \to \C$ be 1-bounded functions  and let $\mu_N$ denote the probability measure
$$
\mu_N := \trecip{N^2}1_{[N]}*1_{-[N]}= \trecip{N}\brac{1-\tfrac{|\cdot|}{N}}_+.
$$ 
Write
$$
\Lambda_N(f_0, f_1, f_2) := \E_{x \in [N]}\E_{y \in [N^2]}\sum_d \mu_N(d)f_0(x,y)f_1(x+d,y) f_2(x,y+d^2).
$$
If $\abs{\Lambda_N(f_0,f_1, f_2)} \geq \delta$, then either $N \ll \delta^{-O(1)}$ or there exists $\alpha \in \T$ and $q \ll \delta^{-O(1)}$ such that $\norm{q\alpha}_{\T} \ll\delta^{-O(1)}/N^2$ and
\begin{equation*}
\E_{x\in[N]} \abs{\E_{y \in [2N^2]} f_2(x, y) e\brac{\alpha y}} \gg \delta^{O(1)}.
\end{equation*}
\end{theorem}
\begin{proof}
We may assume that 
$$
\supp(f_0) \subset [N]\times [N^2],\quad \supp(f_1) \subset [-N, 2N]\times [N^2],\quad \supp(f_2) \subset [N]\times [2N^2].
$$
Define the dual function
$$
G(x,y) := \sum_d \mu_N(d) f_0(x-d,y) f_2(x-d,y+d^2).
$$
Then, by the Cauchy--Schwarz inequality:
$$
\delta  \leq \abs{\recip{N^3}\sum_{x,y} f_1(x,y)G(x,y)} \ll  \brac{\recip{N^3}\sum_{x,y}G(x,y)\overline{G(x,y)}}^{1/2} = \Lambda_N\brac{\overline{f}_0,G, \overline{f}_2}^{1/2}.
$$
Hence, by our first inverse theorem (Theorem \ref{thm:inv-thm-1}), either $N \ll \delta^{-O(1)}$ or there exists $\beta \in \T$ and $q \ll \delta^{-O(1)}$ such that $\norm{q\beta}_{\T} \ll\delta^{-O(1)}/N$ and
\begin{equation*}
\E_{y\in[N^2]} \abs{\E_{x\in[-N,2N]} G(x, y) e\brac{\beta x}} \gg \delta^{O(1)} .
\end{equation*}
Expanding the dual function, changing variables and employing the triangle inequality, we have
$$
 \E_{x\in[N]}\E_{y\in [N^2]}\abs{\sum_d\mu_N(d)  f_2(x,y+d^2)e(\beta d)}  \gg \delta^{O(1)}.
$$

By the popularity principle and Lemma \ref{lem:dual-major}, for each $x \in [N]$ there exist $q_x \ll \delta^{-O(1)}$ and $\alpha_x \in \T$ with $\norm{q_x\alpha_x}_{\T} \ll \delta^{-O(1)}/N^2$ such that 
$$
\E_{x\in[N]}\abs{\E_{y\in [2N^2]} f_2(x,y) e(\alpha_x y)} \gg \delta^{O(1)}.
$$
Pigeon-holing in the major arcs (Lemma \ref{lem:PH-major}), there exists $\alpha \in \T$ and $q \ll \delta^{-O(1)}$ with $\norm{q\alpha}_{\T} \ll \delta^{-O(1)}/N^2$ such that 
$$
\E_{x\in[N]}\abs{\E_{y\in [2N^2]} f_2(x,y) e(\alpha y)} \gg \delta^{O(1)}.
$$
The result follows.
\end{proof}

The previous two inverse theorems concern the global count of configurations. As indicated in \S\ref{sec:outline}, our strategy for proving Theorem \ref{thm:2D-pop-diff} is to focus on counting configurations where the difference parameter is divisible by a small positive integer $q$ and is localised to an interval shorter than $N$. We therefore require an inverse theorem for this localised counting operator. We also take this opportunity to replace correlation with major arcs by convolution with a Fej\'er kernel, which facilitates the energy increment argument in Theorem \ref{thm:energy-1}. Making these adjustments to our inverse theorem is the purpose of the next result.
\begin{theorem}[Large localised count gives a large convolution]\label{thm:large-asym-count}
Let $f_0, f_1, f_2 : \Z^2 \to \C$ be 1-bounded functions with support contained in $[-N_1, N_1]\times [-N_2, N_2]$ where $N_1 \geq N_2^{1/2}$.  Suppose that for some positive integers $q,M$ with $qM \leq N_2^{1/2}$ we have
$$
\abs{\sum_{x,y,d} \mu_M(d)
f_0(x,y)f_1(x+qd,y)f_2(x, y+q^2 d^2)} \geq \delta  N_1N_2.
$$
Then there exists a positive integer $ \tilde q$ with $ \tilde q \ll \delta^{-O(1)}$ such that for any positive integer $ \widetilde M$, either $ \widetilde M \gg \delta^{O(1)} M$ or
\begin{equation}\label{eq:convolved-inverse}
\sum_{x,y} \abs{\sum_{ \tilde d}\mu_{ \widetilde M^2}( \tilde d) f_2(x, y+q^2 \tilde q^2 \tilde d) } \gg \delta^{O(1)} N_1N_2.
\end{equation}
Here $\mu_M$ and $\mu_{ \widetilde M^2}$ denote the Fej\'er kernels defined in \eqref{fejer}.
\end{theorem}
\begin{proof}
By a change of variables, we may average our operator over additional shifts, to deduce that
\begin{multline*}
\Bigl|\sum_{x,y}\E_{x'\in [M]}\E_{ y' \in [M^2]}\sum_d\mu_M(d)
f_0(x+qx',y+q^2y')f_1(x+q(x'+d),y+q^2y')\\
f_2(x+qx', y+q^2 (y' + d^2))\Bigr| \geq \delta  N_1N_2.
\end{multline*}
Hence, on setting
$$
f_i^{x,y,q}(x', y') := f_i(x+qx',y+q^2y'),
$$
we have that
\begin{equation}\label{eq:faq-sum-1}
\Bigl|\sum_{x,y}\Lambda_M\brac{f_0^{x,y,q}, f_1^{x,y,q}, f_2^{x,y,q}}\Bigr| \geq \delta  N_1N_2.
\end{equation}

We note that, since $qM \leq N_2^{1/2} \leq N_1$ and each $f_i$ is supported on $[-N_1, N_1]\times [-N_2, N_2]$, the set of pairs $(x,y)\in \Z^2$ that contribute to \eqref{eq:faq-sum-1} is contained in $[-2N_1, 2N_1]\times [-2N_2, 2N_2]$. Thus, by the popularity principle, there exists a set $\mathcal{P}\subset [-2N_1, 2N_1]\times [-2N_2, 2N_2]$ of size $\gg \delta N_1N_2$ such that 
$$
\abs{\Lambda_M\brac{f_0^{x,y,q}, f_1^{x,y,q}, f_2^{x,y,q}}} \gg \delta \quad \text{for all } (x,y) \in \mathcal{P}. 
$$

By our second counting operator inverse theorem (Theorem \ref{thm:inv-thm-2}), either $M \ll \delta^{-O(1)}$ or, for each pair $(x,y) \in \mathcal{P}$, there exists $\alpha_{x,y} \in \T$ and $q_{x,y} \ll \delta^{-O(1)}$ with $\norm{q_{x,y}\alpha_{x,y}}_{\T} \ll \delta^{-O(1)}/M^2$ and such that 
$$
\E_{x'\in [M]} \abs{\E_{y'\in [2M^2]} f^{x,y, q}_2(x', y') e\brac{\alpha_{x,y} y'}} \gg \delta^{O(1)}.
$$
Notice that we can ignore the possibility that $M \ll \delta^{-O(1)}$, since in this case \textit{all} positive integers $\widetilde M$ satisfy $\widetilde M \gg \delta^{O(1)}M$. Pigeon-holing in the major arcs, as in Lemma \ref{lem:PH-major}, we obtain $\alpha \in \T$ and $ \tilde q \ll \delta^{-O(1)}$ with $\norm{ \tilde q\alpha}_{\T} \ll \delta^{-O(1)}/M^2$ and such that
\begin{equation}\label{eq:faq-sum-2}
\sum_{x,y}\E_{x'\in[M]} \abs{\E_{y'\in[2M^2]} f^{x,y, q}_2(x', y') e\brac{\alpha y'}} \gg \delta^{O(1)} N_1N_2.
\end{equation}

We note that
$$
\abs{e\brac{\alpha y'}-e\brac{\alpha (y'+ \tilde q^2 \tilde d)}} \ll \delta^{-O(1)} | \tilde d|/M^2.
$$ 
Combining this with the fact that the set of pairs $(x,y)$ that contribute to \eqref{eq:faq-sum-2} is contained in $[-2N_1, 2N_1]\times [-2N_2, 2N_2]$, we deduce that for any positive integer $ \widetilde M$, we either have $ \widetilde M \gg \delta^{O(1)} M$ or 
\begin{equation}\label{eq:faq-sum-3}
\sum_{x,y}\E_{x'\in[M]} \abs{\E_{y'\in[2M^2]} f^{x,y, q}_2(x', y') \E_{ \tilde d \in [ \widetilde M^2]}e\brac{\alpha (y'+ \tilde q^2  \tilde d)}} \gg \delta^{O(1)}N_1N_2.
\end{equation}

Instead of averaging the phase in \eqref{eq:faq-sum-3} over shifts of the form $ \tilde q^2  \tilde d$, we would like to change variables and average $f_2^{x,y,q}(x', \cdot)$ over these shifts. We have the identity
\begin{multline*}
\E_{y'\in[2M^2]} f^{x,y, q}_2(x', y') \E_{ \tilde d \in [ \widetilde M^2]}e\brac{\alpha (y'+ \tilde q^2  \tilde d)}\\ = \E_{ \tilde d \in [ \widetilde M^2]}\E_{y'- \tilde q^2 \tilde d\in[2M^2]} f^{x,y, q}_2(x', y'- \tilde q^2 \tilde d) e\brac{\alpha y'}.
\end{multline*}
Hence, estimating the symmetric difference of the sets of summands gives that
\begin{multline*}
\E_{y'\in[2M^2]} f^{x,y, q}_2(x', y') \E_{ \tilde d \in [ \widetilde M^2]}e\brac{\alpha (y'+ \tilde q^2  \tilde d)}\\ = \E_{y'\in[2M^2]}  \E_{ \tilde d \in [ \widetilde M^2]}f^{x,y, q}_2(x', y'- \tilde q^2 \tilde d) e\brac{\alpha y'} +O\brac{\frac{ \tilde q^2
 \widetilde M^2}{M^2}}.
\end{multline*}
It follows that either $ \widetilde M \gg \delta^{O(1)} M$ or we have
\begin{equation*}
\sum_{x,y}\E_{x'\in[M]} \abs{\E_{y'\in[2M^2]} \E_{ \tilde d \in [ \widetilde M^2]}f^{x,y, q}_2(x', y'- \tilde q^2 \tilde d) e\brac{\alpha y'}} \gg \delta^{O(1)}N_1N_2.
\end{equation*}
Applying the triangle inequality gives
\begin{equation*}
\E_{x' \in [M]} \E_{y' \in [2M^2]}\sum_{x,y} \abs{ \E_{ \tilde d \in [ \widetilde M^2]}f^{x,y, q}_2(x', y'- \tilde q^2 \tilde d) } \gg \delta^{O(1)}N_1N_2.
\end{equation*}
Taking a maximum over $x'$ and $y'$, then expanding the definition of each $f_2^{x,y,q}$ and changing variables, we have
$$
\sum_{x,y} \abs{\E_{ \tilde d \in [ \widetilde M^2]} f_2(x, y-q^2 \tilde q^2 \tilde d) } \gg \delta^{O(1)} N_1N_2.
$$
An application of Cauchy--Schwarz to double the $ \tilde d $ variables then yields
\begin{multline*}
\delta^{O(1)}N_1N_2 \ll \sum_{x,y} \E_{ \tilde d_1,  \tilde d_2 \in [ \widetilde M^2]} f_2(x, y-q^2 \tilde q^2 \tilde d_1)\overline{f_2(x, y-q^2 \tilde q^2 \tilde d_2) }\\
\leq  \sum_{x,y}\abs{ \E_{ \tilde d_1, \tilde d_2 \in [ \widetilde M^2]}f_2\brac{x, y+q^2 \tilde q^2(\tilde d_1- \tilde d_2)}}
\end{multline*}
This gives \eqref{eq:convolved-inverse}, on recalling the definition \eqref{fejer} of the Fej\'er kernel $\mu_H$.
\end{proof}
We end this section by deriving the inverse theorem claimed in our introduction (Theorem \ref{thm:inverse}), where there is no occurrence of the Fej\'er kernel $\mu_N$ in the counting operator.
\begin{proof}[Proof of Theorem \ref{thm:inverse}]
Set $M := \ceil{2\delta^{-1} N}$. Due to the support of the $f_i$, if 
$$
f_0(x,y)f_1(x+d,y)f_2(x,y+d^2) \neq 0
$$
then $x \in [N]$, $y \in [N^2]$ and $|d| < N$. It follows that
\begin{multline*}
\abs{f_0(x,y)f_1(x+d,y)f_2(x,y+d^2) - f_0(x,y)f_1(x+d,y)f_2(x,y+d^2)M\mu_M(d)}\\\leq \tfrac{\delta}{2}1_{[N]}(x)1_{[N^2]}(y).
\end{multline*}
Thus the hypothesis \eqref{eq:large-count} implies that
$$
\abs{\Lambda_M(f_0, f_1, f_2)} \gg \delta \tfrac{N^4}{M^4} \gg \delta^{O(1)}.
$$

Applying Theorem \ref{thm:inv-thm-1} and Theorem \ref{thm:inv-thm-2}, either $N \ll \delta^{-O(1)}$ or there exist $\alpha, \beta \in \T$ and $q_1, q_2 \ll \delta^{-O(1)}$ such that $\norm{q_1\alpha} \ll \delta^{-O(1)}/N$ and $\norm{q_2\beta} \ll \delta^{-O(1)}/N^2$, as well as \eqref{eq:inverse-eq} holds. The theorem follows on taking $q:= q_1q_2$.
\end{proof}

\section{Energy increment}\label{sec:energy}
We are now in a position to run an energy increment argument in order to obtain a modified counting operator that is close to our original count: see our outline in \S\ref{sec:outline} for more on this. Green and Tao \cite{GreenTaoNewIa} call the argument underlying the following result a \emph{local Koopman-von Neuman theorem}.
\begin{theorem}[Energy increment to a subprogression]\label{thm:energy-1}
Let $0<\eps \leq 1/2$ and let $f_0, f_1, f_2 : \Z^2 \to \C$ be 1-bounded functions, each with support contained in $[N_1]\times [N_2]$ where $N_1 \geq N_2^{1/2}$.
Then either $N_2 \leq \exp\brac{\eps^{-O(1)}}$, or there exist $q \leq \exp\brac{\eps^{-O(1)}}$ and $M \geq N_2^{1/2}/\exp\brac{\eps^{-O(1)}}$ such that 
\begin{multline}\label{eq:irregular-pair}
\biggabs{\sum_{x,y,d}\mu_{\eps M}(d)f_0(x,y)f_1(x+qd,y)\Bigsqbrac{f_2\brac{x, y+q^2 d^2}- \sum_{\tilde d}\mu_{M^2}(\tilde d)f_2\brac{x, y+q^2\brac{\tilde d+ d^2}}}}\\
 \leq \eps N_1N_2.
\end{multline}
Moreover, we may choose $M$ sufficiently small to ensure that $qM \leq \eps N_2^{1/2}$. Here $\mu_{\eps M}$ and $\mu_{M^2}$ denote Fej\'er kernels, as defined in \eqref{fejer}.
\end{theorem}
\begin{proof}
We perform an iterative procedure. At stage 0, we take $q = q_0 := 1$ and $M = M_0 := \lfloor\eps N_2^{1/2}\rfloor$.  Suppose that at stage $n$ of our iteration we have a positive integers $q = q_n \leq \eps^{-O(n)}$ and $M=M_n \in[ \eps^{O(n+1)} N_2^{1/2}, \eps N_2^{1/2}/q]$ that satisfy the energy lower bound
\begin{equation}\label{eq:energy-bd}
\sum_{x,y} \biggabs{\sum_{\tilde d} f_2(x, y+q^2\tilde d) \mu_{ M^2}(\tilde d)}^2 \gg n\eps^{O(1)}N_1N_2.
\end{equation}
Given this, we query whether \eqref{eq:irregular-pair} holds or not. \textbf{If \eqref{eq:irregular-pair} holds, then the process terminates.} Therefore suppose that \eqref{eq:irregular-pair} does not hold. 

Write 
$$
g_2(x,y) := f_2\brac{x, y}- \sum_{\tilde d}\mu_{M^2}(\tilde d)f_2\brac{x, y+q^2\tilde d},
$$
which is a function with support contained in $[- N_1, N_1]\times [- N_2, N_2]$, since $q^2M^2 \leq  N_2$. Applying the inverse theorem for our asymmetric counting operator (Theorem \ref{thm:large-asym-count}), there exists a positive integer $\tilde q$ with $\tilde q \leq \eps^{-O(1)}$ such that for any positive integer $\widetilde M$, either $\widetilde M \geq \eps^{O(1)} M$ or
\begin{equation}\label{eq:convolved-inverse-1}
\sum_{x,y} \abs{\sum_{\tilde d}\mu_{\widetilde M^2}(\tilde d) g_2(x, y+q^2\tilde q^2\tilde d)}  \gg \eps^{O(1)} N_1N_2.
\end{equation}
Provided that it is not the case that $M \ll \eps^{-O(1)}$, we can find a positive integer $\widetilde M \geq \eps^{O(1)} M$ sufficiently small to guarantee \eqref{eq:convolved-inverse-1}. \textbf{If $M \ll \eps^{-O(1)}$ then the process terminates}.

On writing $f_2^x(y):= f_2(x,y)$ and $\mu_{q,H}(y) := 1_{q\cdot \Z}(y) \mu_H(y/q)$, we may expand the definition of $g_2$ to re-formulate \eqref{eq:convolved-inverse-1} as\footnote{Here we have normalised the $L^1(\Z)$ norm as in \eqref{eq:Lp-norm}.}
\begin{equation*}
\sum_x\norm{f_2^x*\mu_{q^2\tilde q^2, \widetilde M^2} - f_2^x*\mu_{q^2,  M^2}*\mu_{q^2\tilde q^2, \widetilde M^2}}_1  \gg \eps^{O(1)} N_1N_2.
\end{equation*}
One can check that we have the bound
$$
\norm{\mu_{q^2,  M^2}*\mu_{q^2\tilde q^2, \widetilde M^2} - \mu_{q^2,  M^2}}_1 \ll \frac{\tilde q^2 \widetilde M^2}{M^2} + \frac{\tilde q^4 \widetilde M^4}{M^4}.
$$
Decreasing $\widetilde M$ if necessary, but still retaining a lower bound of the form $\widetilde M \geq \eps^{O(1)} M$, we have
\begin{equation*}
\sum_x\norm{f_2^x*\mu_{q^2\tilde q^2, \widetilde M^2} - f_2^x*\mu_{q^2,  M^2}}_1  \gg \eps^{O(1)} N_1N_2.
\end{equation*}
Applying the Cauchy--Schwarz inequality then gives
\begin{equation*}
\sum_x\norm{f_2^x*\mu_{q^2\tilde q^2, \widetilde M^2} - f_2^x*\mu_{q^2,  M^2}}_2^2  \gg \eps^{O(1)} N_1N_2.
\end{equation*}
We wish to replace the above, which measures the energy of a difference of two functions $f-g$, with the difference of their energies $\norm{f}_2^2-\norm{g}_2^2$. From Pythagoreans' theorem, we know that to make this substitution we must show that $f-g$ and $g$ are (approximately) orthogonal. More concretely, writing $L^2$-norms in terms of inner products, we have
\begin{multline}\label{eq:pythag}
 \bignorm{f_2^x*\mu_{q^2\tilde q^2, \widetilde M^2}}_2^2- \bignorm{f_2^x*\mu_{q^2,  M^2}}_2^2 = \bignorm{f_2^x*\mu_{q^2\tilde q^2, \widetilde M^2} - f_2^x*\mu_{q^2,  M^2}}_2^2\\ - 2\Re\bigang{f_2^x*\mu_{q^2,  M^2}, f_2^x*\bigbrac{\mu_{q^2,  M^2}-\mu_{q^2\tilde q^2, \widetilde M^2}}}.
\end{multline}
So we wish to show that the latter inner product is small.

Suppose otherwise, so that the inner product in \eqref{eq:pythag} is at least
$\delta \norm{f_2^x}_2^2$. Then Parseval's identity and the convolution identity give
\begin{multline*}
\delta \norm{f_2^x}_2^2 \leq \Re\bigang{f_2^x*\mu_{q^2,  M^2}, f_2^x*\bigbrac{\mu_{q^2,  M^2}-\mu_{q^2\tilde q^2, \widetilde M^2}}}\\ = \int_{\T}\bigabs{\widehat{f}^x_2(\alpha)}^2 \widehat{\mu}_{q^2, M^2}(\alpha)\sqbrac{\widehat{\mu}_{q^2,  M^2}(\alpha)-\widehat{\mu}_{q^2\tilde q^2, \widetilde M^2}(\alpha)}\intd\alpha\\
\leq \norm{f_2^x}_2^2\sup_\alpha\brac{\widehat{\mu}_{q^2,  M^2}(\alpha)\sqbrac{\widehat{\mu}_{q^2,  M^2}(\alpha)-\widehat{\mu}_{q^2\tilde q^2, \widetilde M^2}(\alpha)}}.
\end{multline*}
By non-negativity of $\widehat\mu_{q^2, M^2}$, there exists $\alpha \in \T$ such that both of the following hold
\begin{equation*}
\widehat{\mu}_{q^2,  M^2}(\alpha) \geq \delta \quad \text{and}\quad \widehat{\mu}_{q^2,  M^2}(\alpha)-\widehat{\mu}_{q^2\tilde q^2, \widetilde M^2}(\alpha) \geq \delta.
\end{equation*}
Yet, we know from summing the geometric series that if
$$
M^{-4}\abs{\hat{1}_{[M^2]}(q^2\alpha)}^2 = \widehat{\mu}_{q^2,  M^2}(\alpha) \geq \delta,
$$
then $\norm{q^2 \alpha}_{\T} \leq \delta^{-1/2}/M^2$. Yet, for such values of $\alpha$, we have
$$
\widehat{\mu}_{q^2\tilde q^2, \widetilde M^2}(\alpha) = 
\bigabs{\E_{\tilde d \in [\widetilde M^{2} ]} e(\alpha q^2 \tilde q^2 \tilde d)}^2 \geq  \bigbrac{1 - \frac{2\pi\tilde q^2 \widetilde M^2}{\delta^{1/2}M^2}}^2 \geq 1 - \frac{4\pi\tilde q^2 \widetilde M^2}{\delta^{1/2}M^2}.
$$
Thus, taking $\delta =  (5\pi\tilde q^2 \widetilde M^2/M^2)^{2/3}$ gives the contradiction that such values of $\alpha$ satisfy $\widehat{\mu}_{q^2\tilde q^2, \widetilde M^2}(\alpha) > 1-\delta$ and $\widehat{\mu}_{q^2, M^2}(\alpha) > 1$.

We have therefore deduced that  
$$
\Re\ang{f_2^x*\mu_{q^2,  M^2}, f_2^x*\brac{\mu_{q^2,  M^2}-\mu_{q^2\tilde q^2, \widetilde M^2}}} \ll \brac{\tfrac{\tilde q \widetilde M}{M}}^{4/3} \norm{f_2^x}_2^2.
$$

Putting all of the above together, we have that 
 \begin{equation*}
\sum_x\brac{\bignorm{f_2^x*\mu_{q^2\tilde q^2, \widetilde M^2}}_2^2 -\norm{ f_2^x*\mu_{q^2,  M^2}}_2^2}  \geq \Omega\brac{ \eps^{O(1)} N_1N_2} - O\brac{\brac{\tfrac{\tilde q \widetilde M}{M}}^{4/3} N_1N_2}.
\end{equation*}
Decreasing $\widetilde M$ if necessary, but still retaining a lower bound of the form $\widetilde M \geq \eps^{O(1)} M$, we obtain the energy increment
$$
\sum_x\bignorm{f_2^x*\mu_{q^2\tilde q^2, \widetilde M^2}}_2^2 \geq \sum_x\bignorm{f_2^x*\mu_{q^2,  M^2}}_2^2 + \Omega\brac{\eps^{O(1)} N_1N_2}.
$$
With this bound in hand, our inductive hypotheses are satisfied (see \eqref{eq:energy-bd}), and our procedure iterates to the next stage.

Since the energy \eqref{eq:energy-bd} is bounded above by $N_1N_2$, our procedure must terminate at some stage $n \ll \eps^{O(1)}$, yielding either $M = M_n \ll \eps^{-O(1)}$ or \eqref{eq:irregular-pair}. Unpacking the rate at which $q = q_n$ and $M = M_n$ change, the former possibility leads to $N \leq \exp(\eps^{-O(1)})$, whilst the latter gives our other desired outcome.
\end{proof}
\section{Obtaining a popular difference}\label{sec:pop-com-diff}
Finally we prove the remaining results claimed in our introduction \S\ref{sec:intro}.
\begin{theorem}[Existence of a popular difference]\label{thm:pop-com-diff}
Let $A \subset [N_1]\times[N_2]$ with $|A| \geq \delta N_1N_2$ and $N_1 \geq N_2^{1/2}$. Given $0<\eps \leq 1/2$, either $N_2 \leq \exp\brac{\eps^{-O(1)}}$ or there exist $q \leq \exp\brac{\eps^{-O(1)}}$ and $M \geq N_2^{1/2}/\exp\brac{\eps^{-O(1)}}$ such that 
$$
\sum_{x,y, d}\mu_{ M}(d)1_A(x,y)1_A(x+qd,y)1_A(x, y+q^2 d^2)\geq \brac{\delta^3-\eps} N_1N_2.
$$
\end{theorem}
\begin{proof}
Let us apply our energy increment argument to $1_A$, so we take $f_0= f_1 = f_2 = 1_A$ in Theorem \ref{thm:energy-1}. This tells us that we either have $N_2 \leq \exp\brac{\eps^{-O(1)}}$ or there exist $q \leq \exp\brac{\eps^{-O(1)}}$ and $M \geq N_2^{1/2}/\exp\brac{\eps^{-O(1)}}$ such that $qM \leq \eps N_2^{1/2}$ and
\begin{multline*}
\biggl|\sum_{x,y,d}\mu_{\eps M}(d)1_A(x,y)1_A(x+qd,y)\\
\biggsqbrac{1_A\brac{x, y+q^2 d^2}- \sum_{\tilde d}\mu_{M^2}(\tilde d)1_A\brac{x, y+q^2\brac{ d^2+\tilde d}}}\biggr|  \leq \eps N_1N_2.
\end{multline*}

Focusing on the second counting operator, we have
\begin{multline*}
\sum_{x,y,d}\mu_{\eps M}(d)1_A(x,y)1_A(x+qd,y)\sum_{\tilde d}\mu_{M^2}\bigbrac{\tilde d}1_A\brac{x, y+q^2\brac{d^2+ \tilde d}}\\
= \sum_{x,y,d}\mu_{\eps M}(d)1_A(x,y)1_A(x+qd,y)\sum_{\tilde d}\mu_{M^2}\brac{\tilde d-d^2}1_A\brac{x, y+q^2\tilde d}.
\end{multline*}
The Lipschitz properties of $\mu_{M^2}$ give that
$$
\abs{\mu_{M^2}\brac{\tilde d-d^2} - \mu_{M^2}\bigbrac{\tilde d}} \ll d^2M^{-4}.
$$
Hence our restriction to $|d|< \eps M$ gives  
\begin{multline*}
\sum_{x,y,d}\mu_{\eps M}(d)1_A(x,y)1_A(x+qd,y)\sum_{\tilde d}\mu_{M^2}\bigbrac{\tilde d}1_A\brac{x, y+q^2\brac{d^2+ \tilde d}}\\
\geq \sum_{x,y,d}\mu_{\eps M}(d)1_A(x,y)1_A(x+qd,y)\sum_{\tilde d}\mu_{M^2}\bigbrac{\tilde d}1_A\brac{x, y+q^2\tilde d} - O\brac{\eps N_1N_2}.
\end{multline*}

Expanding the definition \eqref{fejer} of the Fej\'er kernel, then changing variables,  we have
\begin{multline*}
\sum_{x,y,d}\mu_{\eps M}(d)1_A(x,y)1_A(x+qd,y)\sum_{\tilde d}\mu_{M^2}\bigbrac{\tilde d}1_A\brac{x, y+q^2\tilde d}\\
 \geq \sum_{x,y}\E_{d_1, d_2 \in [\eps M]}\E_{\tilde d_1, \tilde d_2\in [M^2]}1_A\bigbrac{x+qd_1,y+q^2 \tilde d_1}\\ 1_A\bigbrac{x+qd_2,y+q^2 \tilde d_1}1_A\bigbrac{x+qd_1, y+q^2 \tilde d_2}.
\end{multline*}

By an application of H\"older's inequality, sometimes termed the Blakley-Roy inequality\footnote{See this discussion \url{https://mathoverflow.net/questions/189222}.}, we have
\begin{multline*}
\E_{d_1, d_2 \in [\eps M]}\E_{\tilde d_1, \tilde d_2\in [M^2]}1_A(x+qd_1,y+q^2 \tilde d_1)1_A(x+qd_2,y+q^2 \tilde d_1)1_A(x+qd_1, y+q^2 \tilde d_2)\\ \geq \brac{\E_{d \in [\eps M]}\E_{\tilde d\in [M^2]}1_A(x+qd,y+q^2 \tilde d)}^3.
\end{multline*}
If the sum inside the cube is non-zero, then (because $A \subset [N_1]\times [N_2]$) we have 
$$
(x,y) \in (-q\eps M, N_1) \times (-q^2M^2, N_2) \subset (-\eps N_1, N_1)\times (-\eps N_2, N_2).
$$
So, again by H\"older's inequality,
\begin{multline*}
\sum_{x,y}\brac{\E_{d \in [\eps M]}\E_{\tilde d\in [M^2]}1_A(x+qd,y+q^2 \tilde d)}^3\\
\geq N_1^{-2}N_2^{-2}(1+\eps)^{-4} \brac{\sum_{x,y}\E_{d \in [M]}\E_{\tilde d\in [M^2]}1_A(x+qd,y+q^2 \tilde d)}^3\\
\geq  \delta^3N_1N_2(1-\eps)^4.
\end{multline*}

Putting everything together gives  
$$
\sum_{x,y,d}\mu_{\eps M}(d)1_A(x,y)1_A(x+qd,y)1_A(x, y+q^2 d^2) \geq \brac{\delta^3 - O(\eps)}N_1N_2,
$$
as required.
\end{proof}

Taking $N_1 = N_2$ in Theorem \ref{thm:pop-com-diff} yields Theorem \ref{thm:2D-pop-diff}. Taking $N_2 = N_1^2$ yields the following, which is perhaps more natural.
\begin{corollary}
Let $A \subset [N]\times [N^2]$ with $|A| \geq \delta N^3$ and let $0<\eps \leq 1/2$.
Either $N \leq \exp\brac{\eps^{-O(1)}}$ or there exists $d \neq 0$ such that
$$
\hash\set{(x,y) \in A : (x+d, y), (x,y+d^2) \in A} \geq \brac{\delta^3-\eps} N^3.
$$
\end{corollary}
From this we derive the existence of a one-dimensional popular common difference, which is claimed in Theorem \ref{thm:1D-pop-diff}.
\begin{corollary}[Theorem \ref{thm:1D-pop-diff}, re-stated]
Let $A \subset [N]$ with $|A| \geq \delta N$ and let $0<\eps \leq 1/2$.
Either $N \leq \exp\brac{\eps^{-O(1)}}$ or there exists $d \neq 0$ such that
$$
\hash\set{x \in A : x+d, x+d^2 \in A} \geq \brac{\delta^3-\eps} N.
$$
\end{corollary}
\begin{proof}
Define
$$
\tilde A := \set{(x,y) \in [N^{1/2}]\times [N] : x+y \in A}.
$$
Let us bound the size of $\tilde A$. We have:
\begin{equation*}
|\tilde A| \geq \sum_{a \in A \cap (N^{1/2}, N]} \sum_{x\in [N^{1/2}]} 1_{[N]}(a-x) \geq \brac{\delta N - N^{1/2}} N^{1/2} \geq (\delta - \eps) N^{3/2},
\end{equation*}
or else $N \ll \eps^{-O(1)}$.

Applying Theorem \ref{thm:pop-com-diff}, either $N \leq \exp\brac{\eps^{-O(1)}}$ or there exists an integer $d \neq 0$ such that 
\begin{equation*}
 \hash\set{(x,y)\in [N^{1/2}]\times [N] : x+y,\ x+y+d,\ x+y+d^2 \in A} \geq \brac{(\delta-\eps)^3 -\eps} N^{3/2}.
\end{equation*}
Taking a maximum over $x\in [N^{1/2}]$, we obtain some $x_0$ such that
$$
  \hash\set{y\in [N] : x_0+y,\ x_0+y+d,\ x_0+y+d^2 \in A} \geq \brac{(\delta-\eps)^3 -\eps} N.
$$
The result follows on observing that (by the binomial theorem) $(\delta-\eps)^3 \geq \delta^3 - 4\eps$.
\end{proof}
\appendix

\section{Notation}\label{sec:notation}
\subsection{Standard conventions}
We use $\N$ to denote the positive integers.  For real $X \geq 1$, write $[X] = \{ 1,2, \ldots, \floor{X}\}$.  A complex-valued function is \emph{1-bounded} if the modulus of the function does not exceed 1.

We use counting measure on $\Z^d$, so that for $f,g :\Z^d \to \C$ we have
\begin{equation}\label{eq:Lp-norm}
\ang{f,g} := \sum_x f(x)\overline{g(x)}\qquad \text{and}\qquad \norm{f}_{L^p} := \biggbrac{\sum_x |f(x)|^p}^{\recip{p}}.
\end{equation}
Any sum of the form $\sum_x$ is to be interpreted as a sum over $\Z^d$. 
We use Haar probability measure on $\T^d := \R^d/\Z^d$, so that for measurable $F : \T^d \to \C$ we have
$$
\norm{F}_{L^p} := \biggbrac{\int_{\T^d} |F(\alpha)|^p\intd\alpha}^{\recip{p}} = \biggbrac{\int_{[0,1)^d} |F(\alpha)|^p\intd\alpha}^{\recip{p}}
$$
For $\alpha \in \T$ we write $\norm{\alpha}_{\T}$ for the distance to the nearest integer.  

For a finite set $S$ and function $f:S\to\C$, denote the average of $f$ over $S$ by
\[
\E_{s\in S}f(s):=\frac{1}{|S|}\sum_{s\in S}f(s).
\]

Given functions $f,g : G \to \C$ on an additive group with measure $\mu_G$ we define their convolution by 
\begin{equation}\label{convolution}
f*g(x) := \int_G f(x-y) g(y) \intd\mu_G,
\end{equation}
when this makes sense. On the integers $G= \Z$, we take $\mu_G$ to be counting measure.

We define the Fourier transform of $f : \Z^d \to \C$ by 
\begin{equation}\label{Fourier transform}
\hat{f}(\alpha) := \sum_x f(x) e(\alpha \cdot x) \qquad (\alpha \in \T^d),
\end{equation}
again, when this makes sense.  Here $e(\beta)$ stands for $e^{2\pi i \beta}$.

The \emph{difference function} of $f : \Z^d \to \C$ with respect to $h \in \Z^d$ is the function $\Delta_h f : \Z^d \to \C$ given by 
\begin{equation}\label{eq:diff}
\Delta_hf(x) = f(x) \overline{f(x+h)}.
\end{equation}
Iterating, we set
$$
\Delta_{h_1, \dots, h_s} f := \Delta_{h_1} \dots \Delta_{h_s} f.
$$
This allows us to define the \emph{Gowers $U^s$-norm}
\begin{equation}\label{Us def}
\norm{f}_{U^s(\Z^d)} := \brac{\sum_{x, h_1, \dots, h_s\in \Z^d} \Delta_{h_1, \dots, h_s} f(x)}^{1/2^s}.
\end{equation}
When $S \subset \Z^d$ we define the \emph{localised} Gowers $U^s$-norm
\begin{equation}\label{local Us def}
\norm{f}_{U^s(S)} := \norm{f1_S}_{U^s}.
\end{equation}

For a function $f$ and positive-valued function $g$, write $f \ll g$ or $f = O(g)$ if there exists a constant $C$ such that $|f(x)| \le C g(x)$ for all $x$. We write $f = \Omega(g)$ if $f \gg g$.  We sometimes opt for a more explicit approach, using $C$ to denote a large absolute constant, and $c$ to denote a small positive absolute constant.  The values of $C$ and $c$ may change from line to line.

\subsection{Local conventions}
Up to normalisation, all of the above are well-used in the literature. Next we list notation specific to our paper. We have tried to minimise this in order to aid the casual reader.  

For a real parameter $H \geq 1$, we use $\mu_H : \Z \to [0,1]$ to represent the following normalised Fej\'er kernel
\begin{equation}\label{fejer}
\mu_H(h) := \recip{\floor{H}} \brac{1 - \frac{|h|}{\floor{H}}}_+ = \frac{(1_{[H]} * 1_{-[H]} )(h)}{\floor{H}^2}.
\end{equation}
Strictly speaking, ``Fej\'er kernel'' usually refers to the Fourier transform of this function, which is a function on the torus $\T$. We adhere to our unconventional nomenclature since most of our arguments take place in $\Z$.
For a multidimensional vector $h \in \Z^d$ we write
\begin{equation}\label{multidim fejer}
\mu_H(h) = \mu_H(h_1, \dots, h_d) := \mu_H(h_1)\dotsm \mu_H(h_d).
\end{equation}
We observe that this is a probability measure on $\Z^d$ with support in the box $(-H, H)^d$.

Define a \emph{counting operator} on the functions $f_0, f_1, f_2 : \Z^2 \to \C$  by
\begin{equation}\label{counting op}
\Lambda_N(f_0, f_1, f_2) := \E_{x \in [N]}\E_{y \in [N^2]}\sum_{d}\mu_N(d)f_0(x,y)f_1(x+d,y) f_2(x,y+d^2).
\end{equation}
When the $f_i$ are all equal to $f$ we simply write $\Lambda_N(f)$.





\newcommand{\etalchar}[1]{$^{#1}$}

\end{document}